\documentclass[12pt]{scrartcl}

\usepackage{hyperref}
\usepackage{amsfonts}
\usepackage{amsmath}
\usepackage{amssymb}
\usepackage{amsthm}
\usepackage{algorithm}
\usepackage{algorithmic}
\usepackage{graphicx}
\usepackage{color}
\usepackage{subcaption}
\usepackage{tikz,pgfplots}
\usepackage{mathrsfs}
\usepackage[utf8]{inputenc}
\usepackage{xpatch}
\xpatchcmd{\proof}{\itshape}{\normalfont\proofnamefont}{}{}
\newcommand{\proofnamefont}{\bfseries}

\renewcommand{\vec}[1]{\ensuremath \mathbf{\boldsymbol{#1}}}
\newcommand{\N}{\ensuremath{\mathbb{N}}}
\newcommand{\T}{\ensuremath{\mathbb{T}}}
\renewcommand{\S}{\ensuremath{\mathbb{S}}}

\newcommand{\M}{\ensuremath{\mathcal{M}}}

\newcommand{\R}{\ensuremath{\mathbb{R}}}

\newcommand{\E}{\ensuremath{\mathcal{E}}}
\newcommand{\SO}{\ensuremath{\mathrm{SO}(3)}}

\newcommand{\abs}[1]{\ensuremath{\left\vert#1\right\vert}}
\newcommand{\dx}{\mathrm{d}}

\DeclareMathOperator*{\argmin}{argmin}

\DeclareMathOperator*{\lin}{span}

\renewcommand{\d}{\, \mathrm{d}}

\newcommand{\sphere}{\mathbb S}
\newcommand{\norm}[1]{\left\lVert \smash{#1} \right\rVert}

\newcommand\multiset[2]
{\mathchoice{\left(\kern-0.5em{\binom{#1}{#2}}\kern-0.5em\right)}
  {\bigl(\kern-0.3em{\binom{#1}{#2}}\kern-0.3em\bigr)}
  {\bigl(\kern-0.3em{\binom{#1}{#2}}\kern-0.3em\bigr)}
  {\bigl(\kern-0.3em{\binom{#1}{#2}}\kern-0.3em\bigr)}}
\newcommand{\rz}[1]{\ensuremath{\mathord{\mathrm{#1}}}}

\newcommand{\II}{\rz{II}}

\theoremstyle{definition}
\newtheorem{definition}{Definition}[section]
\newtheorem{example}[definition]{Example}

\theoremstyle{plain}
\newtheorem{theorem}[definition]{Theorem}
\newtheorem{lemma}[definition]{Lemma}
\newtheorem{remark}[definition]{Remark}

\newtheorem{corollary}[definition]{Corollary}
\newtheorem{proposition}[definition]{Proposition}
\newtheorem{problem}[definition]{Problem}
\newcommand{\dotprod}[2]{ \left< #1,#2 \right>}

\newenvironment{Lemma}{ \begin{lemma}\normalfont\slshape}{\end{lemma}}

\numberwithin{equation}{section}

\newcommand{\bend}{\hspace*{0ex} \hfill \hbox{\vrule height
    1.5ex\vbox{\hrule width 1.4ex \vskip 1.4ex\hrule  width 1.4ex}\vrule
    height 1.5ex}}

\long\def\symbolfootnote[#1]#2{\begingroup
\def\thefootnote{\fnsymbol{footnote}}\footnote[#1]{#2}\endgroup}

\newcounter{todocounter}
\newcommand{\todo}[2][noisnotdefined]{
 \marginpar{\fcolorbox{black}{yellow}{\footnotesize\textbf{todo}}
 \ifthenelse{\equal{#1}{noisnotdefined}}{}{\textcolor{black}{\newline\tiny #1}}}
 \textbf{\ifthenelse{\equal{#2}{.}}
   {\fcolorbox{red}{white}{\textcolor{red}{$\maltese$}}}{{\textcolor{red}{#2}}}}
 \refstepcounter{todocounter}}

\title{Approximating the Derivative of Manifold-valued Functions}

\date{\today}
\date{}
\author{ Ralf Hielscher\thanks{Faculty of Mathematics and Informatics, TU Bergakademie Freiberg, Germany.\newline E-mail:		\href{mailto:ralf.hielscher@math.tu-chemnitz.de}{ralf.hielscher@math.tu-chemnitz.de}}
	\and Laura Lippert\thanks{Faculty of Mathematics, Chemnitz University of Technology, Germany.\newline E-mail:		\href{mailto:laura.lippert@math.tu-chemnitz.de}{laura.lippert@math.tu-chemnitz.de}}
}

\begin{document}
\maketitle

\begin{abstract}
  We consider the approximation of manifold-valued functions by embedding the
  manifold into a higher dimensional space, applying a vector-valued
  approximation operator and projecting the resulting vector back to the
  manifold. It is well known that the approximation error for manifold-valued
  functions is close to the approximation error for vector-valued functions.
  This is not true anymore if we consider the derivatives of such
  functions. In our paper we give pre-asymptotic error bounds for the
  approximation of the derivative of manifold-valued function. In particular,
  we provide explicit constants that depend on the reach of the embedded
  manifold.
\end{abstract}


\section{Introduction}
Approximating functions $f \colon \Omega \to \mathcal M$ with values in some
$D$-dimensional Riemannian manifold $\mathcal M$ has attracted lots of interest during the
last years. The central challenge is that with $\mathcal M$ not being linear,
the function spaces over $\Omega$ with values in $\mathcal M$ are not linear as
well and hence, all the well established linear approximation methods do not
have a straight forward generalization to manifold-valued functions.

One successful approach to manifold-valued approximation is to consider the
problem locally. Either one maps the function values locally to some linear
approximation space or one uses local averaging based on the geodesic
distance. These approaches allow to generalize subdivision schemes
\cite{Wallner05,XieYu,Gr09,Gro10,ShIt13}, moving least squares \cite{GrSpYu17},
quasi-interpolation \cite{Grohs2013ProjectionBasedQI} or splines
\cite{Shingel08} to the manifold-valued setting.

A different approach is to embed the manifold $\mathcal M$ into some higher
dimensional linear space $\R^{d}$ by a map
$\mathcal E \colon \mathcal M \to \R^{d}$. Note, that according to Nash's
embedding theorem \cite{Nash54} such a mapping always exists and can be guaranteed
to be locally isometric provided the dimension $d$ is at least $D(3D+11)/2$,
where $D$ denotes the dimension of $\mathcal M$. Embedding based approximation
methods can be summarized as follows
\begin{enumerate}
\item Transfer $f \in C(\Omega,\mathcal M$) via the embedding $ \mathcal E$
  into the linear function space $C(\Omega,\R^{d})$. Since often a manifold is described by vectors in $\R^d$, we will use the notation
	$f$ for the function in both function spaces.
\item Use a linear approximation operator
  $I_{\R^{d}} \colon C(\Omega,\R^{d}) \to C(\Omega,\R^{d})$ to find an approximant
  $I_{\R^{d}} f$ in the embedding.
\item Project the resulting $\R^{d}$-valued function back to the manifold
  $I_{\mathcal M} f = P_{\mathcal M} \circ (I_{\R^{d}}  f)$ using some
  projection operator $P_{\mathcal M} \colon \R^{d} \to \mathcal M$.
\end{enumerate}
Because of its generality and simplicity this approach has already been widely
investigated \cite{Grohs2013ProjectionBasedQI,Ga18} and applied \cite{Mo02,SaSe08}.  In
particular it has been shown in \cite{Ga18} that the approximation order of
the embedding based approximation operator $I_{\mathcal M} f$ is the same as
the approximation order of its linear counterpart $I_{\R^{d}}$.  It is
important to note, that the projection operator $P_{\mathcal M}$ is in general
only defined in some neighborhood $U \supset \mathcal M$ of the
manifold. Hence, the pre-asymptotic behavior of the approximation operator
$I_{\mathcal M}$ strongly depends on the size of this neighborhood which is
directly related to the so-called reach of the embedded manifold.

The aim of this paper is to analyze the pre-asymptotic behavior of the
approximation operator $I_{\mathcal M}$ with respect to the reach of the
embedded manifold. While for the error $I_{\mathcal M} f - f$ the reach only
controls the required linear approximation error
$I_{\R^{d}}  f - f$ that allows for a meaningful approximation
$I_{\mathcal M} f$, the situation is completely different for the error of the
derivatives $\d (I_{\mathcal M} f) - \d f$. In this case the chain rule has to be used for the derivative $\d (I_{\mathcal M} f)$ 
of a concatenation of approximation in $\R^d$ and projection on the manifold $\M$. The derivative of the projection on $\M$ leads 
to pre-asymptotic constants, which depend on the reach of the manifold.

Our paper is organized as follows. In section~\ref{sec:submani} we will first
show some general differential geometric properties of submanifolds of
$\R^d$. Most importantly, we identify in Lemma \ref{thm:PM=piE} the projection
operator $P_{\mathcal M}$ with an orthogonal projection in the normal bundle
over the manifold $\mathcal M$. This is only possible within some tubular
neighborhood of the manifold $\M$ that is controlled by the reach of
$\mathcal M$. The relationship between the reach of the manifold $\mathcal M$
and its curvature or second fundamental form is addressed in
section~\ref{sec:embedd-based-prop}. In Theorem~\ref{thm:B} we make use of
these relationships to describe the differential of the projection operator
$P_{\mathcal M}$ in terms of the second fundamental form. In
section~\ref{sec:change-proj-oper} we end up with the main results of this
chapter, that is we show in Theorem~\ref{thm:Psi_bound} that the derivative
$dP_{\mathcal M}(\vec x)$ of the projection operator at some point
$\vec x \in \R^{d}$ satisfies a Lipschitz-condition with respect to $\vec x$. As
our Lipschitz bound is with respect to the Euclidean distance in the embedding
it is more sharp then the bound reported in \cite{BoLiWi19} that relies on the
geodesic distance.

Section~\ref{sec:manifold} is dedicated to manifold-valued approximation. We
show that the approach of using a linear approximation operator on an
embedded manifold $\M$ in $\R^d$ and then projecting back on the manifold inherits the
approximation order of the linear approximation. Our main result is stated in
Theorem~\ref{satz:wichtigableitung} and gives a pre-asymptotic bound for the
approximation error of the first derivative that relies exclusively on the
reach of the embedded manifold. This result is illustrated in
Theorem~\ref{thm:FPS} for a specific approximation operator, the Fourier
partial sum operator.

In the final section~\ref{sec:examples} we consider two real world examples
for approximating manifold valued data. The first example deals with functions
from the two-sphere into the two-dimensional projective space that describe
the dependency between the propagation direction and the polarization
directions of seismic waves. The second example is from crystallographic
texture analysis where the local alignment of the atom lattice is described by
a map with values in the quotient $\SO/\mathcal{S}$ of the rotation group $\SO$
modulo some finite symmetry group $\mathcal{S}$. The derivative of this map has
important connections microscopic and macroscopic material properties.

\section{Submanifolds}
\label{sec:submani}
In this section we will consider smooth compact Riemannian submanifolds $\M$
of $\R^d$. We will show some differential geometric properties of submanifolds
as well as some estimations for the projection $P_\M$ and the differential of
this projection. We will use these results for estimating some approximation
errors in section~\ref{sec:manifold}.

\subsection{The Projection Operator}
\label{sec:projection-operator-1}

Throughout our work we denote by $\mathcal M \subset \R^{d}$ a smooth compact
Riemannian submanifold of $\R^d$. For every point $\vec m\in \M$ we denote the tangent space $T_\vec m\M$ as well as the normal space $N_\vec m\M$. Furthermore, we denote by $P_{\mathcal M} \colon \R^{d} \to \mathcal M$
the projection operator onto $\mathcal M$ defined as the solution of the
minimization problem
\begin{equation}\label{eq:2}
  P_{\mathcal M}(\vec x) = \argmin_{\vec m \in \mathcal M} \norm{\vec x-\vec m}_2.
\end{equation}
In general, this minimization problem does not posses a unique solution for
every $\vec x \in \R^{d}$, since there is an ambiguity to which branch of the manifold the point should be attributed.
However, if we restrict the domain of the definition of
$P_{\mathcal M}$ to some open neighborhood $U \subset \R^{d}$ of
$\mathcal M$ uniqueness can be granted.

In order to find such a neighborhood $U$ we define on the normal bundle
\begin{equation*}
  N\M
  =\{(\vec m,\vec v)\in \R^d\times \R^d:\vec m\in \M,\vec v\in N_\vec m\M\}
\end{equation*}
of $\mathcal M$ the smooth map
\begin{equation*}
  E \colon N\M \to \R^d, \quad
  E(\vec m,\vec v)=\vec m + \vec v,
\end{equation*}
that maps every normal space $N_{\vec m} \mathcal M$ to an affine linear
subspace through $\vec m \in \R^{d}$. Since we assumed $\mathcal M$ to be
compact and smooth, there exist a maximum constant $\tau>0$ such that the
mapping $E$ restricted to the open subset
\begin{equation*}
  V = \{ (\vec m,\vec v) \in N\M: \norm{\vec v}_2<\tau\}
\end{equation*}
of the normal bundle is injective, cf. \cite[6.24]{Lee12}. Setting $U = E(V)$
defines the so-called \emph{tubular neighborhood} of $\mathcal M$ and the
restriction $E \colon V \to U$ becomes a diffeomorphism. The constant $\tau$
is commonly called \emph{reach} and its inverse $1/\tau$ is the
\emph{condition number} of the manifold. The reach $\tau$ is affected by two
factors: the curvature of the manifold and the width of the narrowest
bottleneck-like structure of $\M$, which quantifies how far $\M$ is from
being self-intersecting. An estimate on the relationship between the reach and
the curvature of the manifold $\mathcal M$ will be given in
Lemma~\ref{lemma:curvatureReach}.

Using the mapping $E$ we may now give an explicit definition of the projection
operator $P_{\mathcal M}$.
\begin{Lemma}
  \label{thm:PM=piE}
  Let $\vec u \in U$ and let $\pi \colon N \mathcal M \to \mathcal M$,
  $(\vec m,\vec n) \mapsto \vec m$ be the canonical projection operator. Then
  \begin{equation*}
    P_\M(\vec u) = \pi \circ E^{-1}(\vec u)
  \end{equation*}
  is the unique solution of the minimization problem \eqref{eq:2}.
\end{Lemma}
\begin{proof}
  Let $\vec u\in U$ and $P_\M(\vec u)=\vec m\in \M$. We show that
  $\vec u-\vec m\in N_\vec m\M$. We assume the opposite and decompose
  $\vec u-\vec m$ in one part in $N_\vec m\ M$ and a part $\vec t$ in
  $T_\vec mM$.  Then there is a curve $\gamma(s)$ in $\M$ with
  $\gamma(0)=\vec m$ and $\dot{\gamma}(0)=\vec t$.  If we go along this curve,
  we obtain for sufficient small $\epsilon >0$ that
  $\vec u-\gamma(\epsilon)<\vec u-\gamma(0)=\vec u-\vec m$.  That is a
  contradiction to the definition of $P_\M$.  Since the projection $P_\M$
  should be unique, we have to show that $\pi\circ E^{-1}$ is also unique.
  For that reason we assume that for $\vec u\in U$ there holds
  $\pi\circ E^{-1}=\vec m\in \M$ and $\pi\circ E^{-1}=\vec m'\in \M$. This
  would imply $\vec u=\vec m+\vec v=\vec m'+\vec v'$ with
  $\vec v\in N_\vec m\M$ and $\vec v\in N_{\vec m'}\M$.  That is a
  contradiction to the uniqueness of $E^{-1}$ in the tubular neighborhood $U$.
\end{proof}
Let us illustrate this by a simple example.
\begin{example}\label{ex:Sd}
Let the manifold $\M$ be the $(d-1)$-dimensional unit sphere,
embedded in $\R^d$. These manifolds can be described by
\begin{equation*}
\S^{d-1}=\left\{\vec x\in \R^{d}:\|\vec x\|_2=1\right\}.
\end{equation*}
The projection $P_{\S^{d-1}}$ easily reads
\begin{equation*}
P_{\S^{d-1}}: \R^{d}\backslash\{\vec 0\} \rightarrow \S^{d-1}, \quad \quad P_{\S^{d-1}}(\vec x)=\frac{\vec x}{\|\vec x\|_2}.
\end{equation*}
This map is well-defined and smooth.
\end{example}

\subsection{Curvature and Reach of Submanifolds}
\label{sec:embedd-based-prop}

For any point $\vec m \in \mathcal M \subset \R^{d}$ on the manifold we can
decompose $\R^{d}$ as the direct sum $\R^d=T_\vec m\M\oplus N_\vec m\M$ of the
tangential space $T_\vec m\M$ and the normal space $N_\vec m\M$. Let us denote
by $P_T\colon\R^d \to T_\vec m\M$ and $P_N\colon \R^d \to N_\vec m\M$ the
corresponding orthogonal projections. Then the canonical connection $\nabla$
on $\R^{d}$ defines a connection $\nabla^{\mathcal M}$ on $\mathcal M$ by
\begin{equation}
  \label{eq:7}
  \nabla^{\mathcal M}_{\vec X}
  = P_{T} \nabla_{\vec X}(P_{T} \vec Y) + P_{N} \nabla_{\vec X}(P_{N} \vec Y),
\end{equation}
where $\vec X \colon \mathcal M \to T \mathcal M$ is a tangential and
$\vec Y \colon \mathcal M \to \R^{d}$ a general vector field on
$\mathcal M$.

If $\vec Y$ is a tangential vector field as well, the first summand
$P_{T} \nabla_{\vec X}(P_{T} \vec Y) = P_{T} \nabla_{\vec X} \vec Y$ in
\eqref{eq:7} is just the Levi-Cevita-connection on $\mathcal M$, whereas its
orthogonal complement
\begin{equation*}
  \II(\vec X,\vec Y) = P_{N}(\nabla_{\vec X} \vec Y)
\end{equation*}
is the \emph{second fundamental form} on $\mathcal M$.

We call a vector field $\vec Y \colon \mathcal M \to \R^{d}$ parallel along a
curve $\gamma$ if $\nabla^{\mathcal M}_{\dot \gamma} \vec Y = 0$. For a
geodesic $\gamma$ with $\gamma(0) = \vec m$,
$\dot \gamma(0) = \vec t \in T_{\vec m} \mathcal M$ and an arbitrary vector
$\vec y \in T_{\vec m} \mathcal M \oplus N_{\vec m} \mathcal M = \R^{d}$ we
shall use the abbreviation
\begin{equation*}
  \nabla_{\vec t} \vec y = \nabla_{\vec t} \vec Y(0)
\end{equation*}
where $\vec Y$ is the parallel transport of the vector $\vec y$ along the
curve $\gamma$.

For a fixed point $\vec m \in \mathcal M$ and a normal direction
$\vec n \in N_{\vec m} \mathcal M$ we define the operator
$\vec B_{\vec n} \colon T_{\vec m}\mathcal M \to T_{\vec m}\mathcal M$ on the
tangent space by
\begin{equation}
  \label{eq:Bn}
  \dotprod{\vec B_{\vec n} \vec x}{\vec y} = \dotprod{\vec n}{\nabla_{\vec x} \vec y},
  \quad \vec x,\vec y \in T_{\vec m} \mathcal M.
\end{equation}

We may also express $B_{\vec n} \vec x$ as the tangential part of the
covariant derivative of $\vec n$ in direction $\vec x$.

\begin{lemma}
  \label{lem:BnX=nablaXn}
  Let $\vec n \in N_{\vec m}\mathcal M$ be a normal and
  $\vec x \in T_{\vec m} \mathcal M$ a tangential vector. Then
  \begin{equation*}
    \vec B_\vec n \vec x = - P_T\nabla_{\vec x}\vec n.
  \end{equation*}
\end{lemma}
\begin{proof}
  Let $\gamma$ be a geodesics in $\M$ with $\gamma(0)=\vec m$ and
  $\dot{\gamma}(0) = \vec x$ and let $\vec N$ be the parallel transport of
  $\vec n$ along $\gamma$. Let furthermore, $\vec Y$ be an arbitrary tangent
  vector field parallel along $\gamma$. Then we have
  \begin{equation*}
    0
    =\frac{\d}{\d s}\langle\vec N(s),\vec Y(s)\rangle|_{s=0}
    =\langle\nabla_{\vec x}\vec N(0),\vec Y(0)\rangle
    +\langle\vec N(0),\vec \nabla_{\vec x}\vec Y(0)\rangle.
  \end{equation*}
  This yields
  \begin{equation*}
    \langle\vec B_\vec n\vec x,\vec y\rangle
    =\langle\vec n,\nabla_{\vec x}\vec y\rangle
    =-\langle\nabla_{\vec x}\vec n,\vec y\rangle.
  \end{equation*}
  Since the vector field $\vec Y$ was arbitrary, this yields the assertion.
\end{proof}

The operator $\vec B_{\vec n}$ describes the extrinsic curvature of the
manifold in the point $\vec m $ and the normal direction $\vec n$. Its norm is
bounded by the condition number $\frac{1}{\tau}$ of $\mathcal M$. More
precisely the following result is shown in \cite[Proposition 6.1]{NSW08}.

\begin{lemma}
  \label{lemma:curvatureReach}
Let $\tau$ be the reach of $\M$, $\vec m\in \M$ be an arbitrary point on the manifold and $\vec n \in N_{\vec m}\mathcal M$ be
  a normal vector. Then the operator $B_{\vec n}$ defined in \eqref{eq:Bn} is
  symmetric and bounded by $\frac{1}{\tau}$, i.e., we have for tangential
  vectors $\vec x,\vec y \in T_{\vec m} \mathcal M$ the inequality
  \begin{equation}
    \label{eq:2formbound}
    \dotprod{\vec B_{\vec n} \vec x}{\vec y}
    \leq \frac{1}{\tau} \norm{\vec n} \norm{\vec x} \norm{\vec y}.
  \end{equation}
\end{lemma}

The next lemma bounds the covariant derivative of parallel vector fields by
the condition number $\frac{1}{\tau}$ of the manifold.

\begin{lemma}\label{lem:bound_diff}
  Let $\vec Y$ be a parallel vector field along a geodesic $\gamma$ in
  $\M$. Then its covariant derivative in $\R^{d}$ is bounded by
  \begin{equation*}
   \norm{\nabla_{\dot{\gamma}}\vec Y}_2\leq \tfrac 1\tau \norm{\vec Y}_{2}\norm{\dot{\gamma}}_2.
  \end{equation*}
\end{lemma}
\begin{proof}
  Let $\vec Y = \vec T + \vec N$ be the decomposition of $\vec Y$ into a
  tangent vector field $\vec T$ and a normal vector field $\vec N$. Since
  $\vec Y$ is parallel along $\gamma$ we have
  \begin{align*}
    0
    = \nabla_{\dot \gamma}^{\mathcal M} \vec Y
    = P_{T} \nabla_{\dot \gamma} \vec T + P_{N} \nabla_{\dot \gamma} \vec N
  \end{align*}
  and, hence,
  \begin{equation*}
    \nabla_{\dot \gamma} \vec Y
    = P_{N} \nabla_{\dot \gamma} \vec T
    + P_{T} \nabla_{\dot \gamma} \vec N.
  \end{equation*}
Let $\vec n = P_{N} \nabla_{\dot \gamma(s)} \vec T(s)$. Then we obtain by Lemma
  \ref{lemma:curvatureReach}
  \begin{equation*}
    \norm{\vec n}^{2}
    = \dotprod{\vec n}{\nabla_{\dot \gamma(s)} \vec T(s)}
    = \dotprod{\vec B_{\vec n} \dot \gamma(s)}{\vec T(s)}
    \le \tfrac{1}{\tau} \norm{\dot \gamma} \norm{\vec n} \norm{\vec T(s)}.
  \end{equation*}
  For the tangential part we have by Lemma \ref{lem:BnX=nablaXn}
  \begin{equation*}
    \norm{P_{T} \nabla_{\dot{\gamma}(s)}\vec N(s)}_2 = \norm{\vec B_{\vec N(s)}
    \dot{\gamma}(s)}_2\leq \tfrac 1\tau \norm{\dot \gamma}\norm{\vec N(s)}.
  \end{equation*}
  The assertion follows now from Parsevals inequality.
\end{proof}

\subsection{The Differential of the Projection Operator.}
\label{sec:diff-proj-oper}

The differential
$\d P_{\mathcal M}(\vec m) \colon \R^{d} \to T_{\vec m} \mathcal M$ of the
projection $P_{\mathcal M} \colon \R^{d} \to \mathcal M$ is especially easy to
compute at points $\vec m \in \mathcal M$ on the manifold. In this case it is
simply the linear projection
$P_{T_{\vec m} \mathcal M} \colon \R^{d} \to T_{\vec m} \mathcal M$ onto the
tangential space attached to $\vec m$, i.e.
\begin{equation}
  \label{eq:dPMm}
  \d P_{\mathcal M}(\vec m) = P_{T_{\vec m} \mathcal M}.
\end{equation}
We can verify this by observing that for normal vectors $\vec n\in N_\vec m\M$
we have
\begin{equation*}
  \d P_\M(\vec m)\,\vec n
  =\lim_{h\to 0}\frac{P_\M(\vec m+h\vec n)-P_\M(\vec m)}{h}
  =\lim_{h\to 0}\frac{\vec m-\vec m}{h}=\vec 0,
\end{equation*}
while for tangent vectors $\vec t\in T_\vec m\M$ we obtain
\begin{align*}
  \d P_\M(\vec m)\,\vec t
  &=\lim_{h\to 0}\frac{P_\M(\vec m+h\vec t)-P_\M(\vec m)}{h}
  =\lim_{h\to 0}\frac{\vec m+\exp(h\vec t)-\vec m}{h}\\
  &=\lim_{h\to 0}\frac{\exp(h\vec t)}{h}=\vec t,
\end{align*}
where $\exp$ denotes the exponential map to the manifold.

The differential $\d P_{\mathcal M}(\vec m + \vec v)$,
$\vec v \in N_{\vec m} \mathcal M$ at a point not in the manifold is a little
bit more tricky. We start by observing that the tangential
$T_{(\vec m,\vec v)} N\mathcal M \subset \R^{2d}$ of the normal bundle at a point
$(\vec m,\vec v) \in N \mathcal M$ is
\begin{align*}
  T_{(\vec m,\vec v)} N\mathcal M
  &= \{ (\vec 0,\vec n) \mid \vec n\in N_{\vec m} \mathcal M \}
  \oplus \{(\vec t,\nabla_{\vec t} \vec v) \mid \vec t \in T_{\vec m} \mathcal  M\} \\
  &= \{ (\vec t, \vec u) \mid
    \vec t \in T_{\vec m} \mathcal M,
    P_{T} \vec u = \nabla_{\vec t} \vec v\}.
\end{align*}
The following lemma describes the differential $\d P_\M(\vec m+\vec v)$.

\begin{lemma}
  \label{lemma:dPM}
  Let $\vec m \in \M$ be an arbitrary point on the manifold $\M$ and $\vec v \in N_{\vec m}\mathcal M$ be a normal
  vector with $\norm{\vec v}_2<\tau$, i.e. $\vec m +\vec v$ is in the tubular
	neighborhood of $\M$. Then the derivative $\dx P_\M(\vec m+\vec v)$ satisfies for every
  tangent direction $\vec t\in T_{\vec m}\M$,
  \begin{equation*}
    \left(\dx P_\M(\vec m+\vec v)\right)
    \left(\vec t+\nabla_{\vec t}\vec v\right)
    =\vec t.
  \end{equation*}
  while it vanishes for any normal direction  $\vec n\in N_\vec m\M$, i.e.
  \begin{equation*}
    \dx P_\M(\vec m+\vec v)\,\vec n=\vec 0.
  \end{equation*}
\end{lemma}
\begin{proof}
  According to Lemma~\ref{thm:PM=piE} we have $P_\M = \pi \circ E^{-1}$, where
  $\pi \colon N\mathcal M \to \mathcal M$, $(\vec m,\vec v) \mapsto \vec m$ is
  the projection operator. Its differential at the point $(\vec m,\vec v) \in
  N\mathcal M$ is the projection
  \begin{equation*}
    \d \pi(\vec m,\vec v)
    \colon T_{(\vec m,\vec v)} N \mathcal M \to T_{\vec m} \mathcal M, \quad
    (\vec t, \vec u) \mapsto \vec t.
  \end{equation*}
  The differential of the mapping $E \colon N \mathcal M \to \R^{d}$ in a
  point $(\vec m,\vec v) \in N\mathcal M$ is given by
  \begin{equation*}
    \d E(\vec m,\vec v) \colon T_{(\vec m,\vec v)} N \mathcal M \to \R^{d},
    \quad
    (\vec t, \vec u) \mapsto \vec t + \vec u.
  \end{equation*}
  Since $\vec m+\vec v$ is within the tubular neighborhood of
  $\mathcal M$, $E$ is invertible in some neighborhood of
  $\vec m+\vec v$. Then $\d E(\vec m,\vec v)$ is invertible as well and we
  have for any normal vector $\vec n \in N_{\vec m} \mathcal M$
  \begin{equation*}
    \d E^{-1}(\vec m + \vec v)\, \vec n = (\vec 0, \vec n)
  \end{equation*}
  and for any tangent vector $\vec t \in T_{\vec m} \mathcal M$
  \begin{equation*}
    \d E^{-1}(\vec m + \vec v)\, (\vec t + \nabla_{\vec t} \vec v)
    = (\vec t, \nabla_{\vec t} \vec v).
  \end{equation*}
  Together with the chain rule this implies the assertion.
\end{proof}

The image of $\dx P_\M(\vec m+\vec v)$ is contained
in the tangential space $T_\vec m\M$, especially $\dx P_\M(\vec m+\vec v)$ is
the projection $P_{T_\vec m\M}$ up to a factor matrix. We will write this
linear operator $\dx P_\M(\vec m+\vec v)$ in another way, to see the
difference to the linear operator $\dx P_\M(\vec m)$.

\begin{theorem}
  \label{thm:B}
  Let $\vec m \in \M$ be a point on the manifold, let
  $\vec v \in N_{\vec m} \mathcal M$ be a normal vector with
  $\norm{\vec v} < \tau$ and let
  $\vec B_{\vec v} \colon T_{\vec m}\mathcal M \to T_{\vec m}\mathcal M$ be
  the symmetric operator defined in \eqref{eq:Bn}, extended to
  $\vec B_{\vec v} \colon \R^d \to \R^d$ by $\vec B_\vec v\vec n=\vec 0$ for
  all normal vectors $\vec n\in N_\vec m\M$. Then the derivative of the
  projection operator $P_\M$ satisfies
  \begin{equation*}
    \dx P_\M(\vec m+\vec v)
    =P_{T_{\vec m} \M}(\vec I - \vec B_{\vec v})^{-1}
    =\dx P_\M(\vec m) - \vec B_\vec v\left(\vec I+\vec B_{\vec v}\right)^{-1},
  \end{equation*}
  where $\vec I \colon \R^d \to \R^d$ is the identity.
\end{theorem}
\begin{proof}
  Using Lemma \ref{lemma:dPM} we obtain for all tangential vectors
  $\vec t\in T_\vec m\M$,
  \begin{align*}
    P_{T_\vec m\M}\vec t
    &=\vec t=\left(\dx P_\M(\vec m+\vec v)\right)\left(\vec t+\nabla_{\vec t} \vec v\right)\\
    &=\left(\dx P_\M(\vec m+\vec v)\right)\left(\vec t-\vec B_\vec v\vec t\right)
      =\left(\dx P_\M(\vec m+\vec v)\right)\left(\vec I-\vec B_\vec v\right)\vec t.
  \end{align*}
  and for all normal vectors $\vec n\in N_\vec m\M$,
  \begin{equation*}
    \vec 0  = P_{T_\vec m\M}\vec n=\left(\dx P_\M(\vec m+\vec v)\right)\left(\vec I-\vec B_\vec v\right)\vec n.
  \end{equation*}
  Consequently, we have
  \begin{equation*}
    P_{T_\vec m\M} = \dx P_\M(\vec m+\vec v) \left(\vec I-\vec B_\vec v\right).
  \end{equation*}
  By our assumption and \eqref{eq:2formbound} we have
  $\norm{\vec B_\vec v}\leq\frac{1}{\tau} \norm{\vec v}<1$ and hence, the
  operator $\vec I-\vec B_\vec v$ is invertible. This yields the first part of
  the assertion. For the second part we use \eqref{eq:dPMm} and compute
  \begin{align*}
    \dx P_\M(\vec m+\vec v)
    &=P_{T_\vec m\M}\left(\vec I-\vec B_\vec v\right)^{-1}\\
    &=P_{T_\vec m\M}\left(\vec I+\vec B_\vec v\left(\vec I-\vec B_\vec v\right)^{-1}\right)\\
    &=\dx P_\M(\vec m)+P_{T_\vec m\M}\vec B_\vec v\left(\vec I-\vec B_\vec v\right)^{-1}\\
    &=\dx P_\M(\vec m)+\vec B_\vec v\left(\vec I-\vec B_\vec v\right)^{-1},
  \end{align*}
  where the last equality follows from the fact that the image of
  $\vec B_\vec n$ is in the tangent space $T_{\vec m}\M$, so the projection on
  $T_{\vec m}\M$ is unnecessary.
\end{proof}
We consider again the manifold from example~\ref{ex:Sd}.
\begin{example}
  For $\M=\S^{d-1}\subset \R^d$ any normal vector $\vec v \in N_{\vec
    m}\S^{d-1}$ has the representation $\vec v = v\,\vec m$.  Let
  $\{\vec t_i\}_{i=1}^{d-1}\subset T_{\vec m}\S^{d-1}$ be an orthonormal basis of
  $T_m\S^{d-1}$. Then $\nabla_{\vec t_i} \vec m = \vec t_{j}$ and hence
  \begin{equation*}
    \vec B_\vec v=-v\sum_{i=1}^{d-1}{\vec t_i \vec t_i^\top}.
  \end{equation*}
  By Theorem~\ref{thm:B} and the orthonormality of
  $\{\vec m\}\cap \{\vec t_i\}_{i=1}^{d-1}$ of we obtain for $v > -1$ and
  $\vec x = \vec m + v \vec m$,
  \begin{equation*}
    \label{eq:11}
    \d P_{\S^{d-1}}(\vec x)
    = \frac{1}{1+v}\sum_{i=1}^{d-1}{\vec t_i \vec t_i^\top}
    = \frac{1}{\norm{\vec x}_2}\,
    \left(\vec I_{d\times d}-\frac{\vec x}{\norm{\vec x}_2}\left(\frac{\vec x}{\norm{\vec x}_2}\right)^\top\right).
  \end{equation*}
\end{example}

\subsection{Deviation of the Projection Operator}
\label{sec:change-proj-oper}

In this section we are interested in the change of the derivative
$\dx P_\M(\vec m)$ of the projection operator for small deviations of
$\vec m$. We shall show that for two points $\vec m$ and $\vec z$ on $\M$ and
$\vec v\in N_\vec m\M$ with $\norm{\vec v}_2<\tau$ we can bound the difference
$\norm{\d P_\M(\vec m+\vec v)-\d P_\M(\vec z)}_2$ by a multiple of the
Euclidean distance $\norm{\vec m+\vec v-\vec z}_2$.

As usual we start with the case that both points are on the manifold.
According to \cite[Lemma 6]{BoLiWi19} the difference of the differentials is
then bounded by
\begin{equation*}
  \norm{\d P_\M(\vec m)-\d P_\M(\vec z)}_2
  \le \frac{1}{\tau}\, d(\vec m,\vec z),
\end{equation*}
where $d(\vec m,\vec z)$ denotes the geodesic distance between the points
$\vec m, \vec z \in \mathcal M$. If the Euclidean distance between the two
points is bounded by $\norm{\vec m - \vec z}_2 \le 2 \tau$ we have by
\cite[Lemma 3]{BoLiWi19} and the fact that $\arcsin(x)\leq \frac{\pi}{2}x$
for $0\leq x\leq 1$, the following estimate between geodesic distance and
Euclidean distance in the embedding
\begin{equation}
  \label{eq:XX}
  d(\vec m, \vec z) \le \frac\pi 2 \norm{\vec m - \vec z}_2,
\end{equation}
which leads to the local estimate
\begin{equation*}
  \norm{\d P_\M(\vec m)-\d P_\M(\vec z)}_2
  \le \frac{\pi}{2 \tau} \norm{\vec m - \vec z}_2.
\end{equation*}
In the following Theorem we prove a sharper and global bound for this
difference.

\begin{theorem}
  \label{thm:C2Teil1}
  For all $\vec m,\vec z\in \M$ the difference between the projection
  operators $P_{T_\vec m\M}$ and $P_{T_\vec z \M}$  onto the respective
  tangential spaces is bounded by
  \begin{equation*}
    \left\|P_{T_\vec m\M}-P_{T_\vec z\M}\right\|_2 \leq \frac 1\tau \norm{\vec m -\vec z}_2.
  \end{equation*}
\end{theorem}

\begin{proof}
  First of all we note that for $\norm{\vec m-\vec z}_2\ge 2\,\tau$ the
  assertion is immediately satisfied since
  $\left\|P_{T_\vec m\M}-P_{T_\vec z\M}\right\|_2 \leq 2$ independently of
  $\vec m,\vec z\in \M$. We may therefore assume
  $\norm{\vec m-\vec z}_2 < 2\,\tau$ for the rest of the proof.

  In order to estimate the difference between the two projection operators we
  consider a geodesic $\gamma$ with $\gamma(0)=\vec m$, $\gamma(t) = \vec z$
  and $\norm{\dot \gamma}_2=1$. Furthermore, we consider an orthonormal basis
  $\{\vec t_i\}_{i=1}^D$ in $T_\vec m\M$ and an orthonormal basis
  $\{\vec n_j\}_{j=1}^{d-D}$ in $N_\vec m\M$. The parallel transport of these
  basis vectors along $\gamma$ defines a rotation $\vec R \in \mathrm{SO}(d)$
  that maps the tangent space $T_{\vec m}\M$ onto the tangent space
  $T_{\vec z}\M$. Using the rotation $\vec R$ we may rewrite the difference
  between the projection operators as
  \begin{equation*}
    P_{T_\vec m\M} - P_{T_\vec z\M}
    = P_{T_\vec m\M} - \vec R P_{T_\vec m\M} \vec R^{T}.
  \end{equation*}
  By Lemma~\ref{lemma:commutator} in the appendix we obtain
  \begin{equation}
    \label{eq:5}
    \norm{P_{T_\vec m\M} - P_{T_\vec z\M}}_{2} = \norm{P_{T_\vec m\M} \vec R -
      \vec R P_{T_\vec z\M}}_{2} \le \norm{\vec I - \vec R}_{2}
  \end{equation}
  and hence, it suffices to bound for any normalized $\vec x \in \R^{d}$
  \begin{equation}
    \label{eq:6}
    \norm{(\vec I - \vec R) \vec x}_{2}^{2}
    = 2 - 2\dotprod{\vec x}{\vec R \vec x}.
  \end{equation}

  By definition $\vec R \vec x$ is the result of the parallel transport of
  $\vec x$ along the curve $\gamma$ in $\gamma(t) = \vec z$. Let us denote by
  $\vec X(s)$ the parallel transport of $\vec x$ along $\gamma$ for all
  times $s \in [0,t]$. Viewing $s \mapsto \vec X(s)$ as a curve on
  $\mathbb S^{d-1}$ with velocity bounded according to
  Lemma~\ref{lem:bound_diff} by
  $\norm{\dot{\vec X}(s)} = \norm{\nabla_{\dot{\gamma}(s)}\vec X(s)}_2 \leq
  \frac{1}{\tau}$, we conclude that
  \begin{equation}
    \label{eq:3}
    \angle(\vec X(\eta),\vec X(\xi))
    \leq \frac{1}{\tau}|\eta-\xi|, \quad \eta,\xi \in [0,t].
  \end{equation}

  Since $\gamma$ is a geodesic we can set in \eqref{eq:3},
  $\vec X = \dot \gamma$. As $\abs{\eta - \xi} \le t$ and $t$ is the geodesic
  distance between $\vec z$ and $\vec m$ we can use \eqref{eq:XX} and our
  assumption $\norm{\vec m-\vec z}_2 < 2 \tau$ to bound the right hand side of
  \eqref{eq:3} by
  \begin{equation*}
    \angle(\dot{\gamma}(\xi),\dot{\gamma}(\eta))
    \le \frac{1}{\tau} \abs{\eta - \xi}
    \le \frac{t}{\tau}
    \le \frac{\pi}{2\tau} \norm{\vec z - \vec m}_{2} \le \pi.
  \end{equation*}
  Making use of the monotonicity of the cosine this implies
  \begin{equation}
    \label{eq:10}
    \cos \angle(\dot{\gamma}(\xi),\dot{\gamma}(\eta))
    > \cos \tfrac{\xi-\eta}{\tau}, \quad
    \xi, \eta \in [0,t].
  \end{equation}
  Considering again the general vector field $\vec X$ we use
  \eqref{eq:3} and \eqref{eq:10} to bound \eqref{eq:6} by
  \begin{align*}
    2-2\dotprod{\vec X(0)}{\vec X(t)}
    &= 2-2\,\cos(\angle (\vec X(0),\vec X(t)))\\
    &\leq 2 - 2\,\cos \tfrac{t}{\tau}
      =\frac{1}{\tau^2} \int_{0}^t\int_{0}^t
      \cos \tfrac{\xi-\eta}{\tau}\,\dx \eta\,\dx \xi\\
    &\leq \frac{1}{\tau^2} \int_{0}^t\int_{0}^t \cos
      \angle(\dot{\gamma}(\xi),\dot{\gamma}(\eta)) \,\dx \eta\,\dx \xi\\
    &= \frac{1}{\tau^2} \int_{0}^t\int_{0}^t\langle\dot{\gamma}(\xi),\dot{\gamma}(\eta)\rangle\,\dx \eta\,\dx \xi
      = \frac{1}{\tau^2} \norm{\vec m-\vec z}_2^2.
  \end{align*}
  In combination with \eqref{eq:5} and \eqref{eq:6} this proves
  \begin{equation*}
    \left\|P_{T_\vec m\M}-P_{T_\vec z\M}\right\|_2 \leq \frac 1\tau \norm{\vec m -\vec z}_2.\qedhere
  \end{equation*}
\end{proof}
Using the example of the unit circle it can be easily verified that our new
bound is sharp.

So far we bounded the variation of the projection operator for points on the
manifold. For the general case that only one point is on the manifold we have
the following result.

\begin{theorem}
  \label{thm:Psi_bound}
  Let $\vec m,\vec z\in \M$ and $\vec v\in N_\vec m\M$ with
  $\norm{\vec v}_2<\tau$. Then
  \begin{align*}
    \norm{\dx P_\M(\vec m+\vec v)-\dx  P_\M(\vec z)}_2
    &\le \frac{1}{\tau} \norm{\vec m - \vec z}_{2}
    + \frac{1}{\tau- \norm{\vec v}_{2}} \norm{\vec v}_{2}\\
    &\leq\left(\frac 2\tau +\frac{1}{\tau-\norm{\vec v}_2}\right)
    \norm{\vec m+\vec v-\vec z}_2.
  \end{align*}
\end{theorem}
\begin{proof}
  Using Theorem~\ref{thm:B} and Theorem~\ref{thm:C2Teil1} we find
  \begin{align*}
    \label{eq:split}
    \norm{\dx P_M(\vec m+\vec v) - \dx  P_M(\vec z)}_2
    &\le \norm{\dx P_M(\vec m+\vec v)-\dx  P_M(\vec m)}_2
      + \norm{ P_{T_\vec m\M}-P_{T_\vec z\M}}_2 \\
    &\le \norm{\vec B_\vec v\left(\vec I+\vec B_\vec v\right)^{-1} }_2
    + \tfrac{1}{\tau} \norm{\vec m - \vec z}_{2}.
  \end{align*}
  From Lemma~\ref{lemma:curvatureReach} we know that $\norm{\vec B_\vec v} \leq
  \frac{\norm{\vec v}_2}{\tau}<1$. This allows us to bound the second term by
  \begin{equation*}
    \norm{\vec B_{\vec v}\left(\vec I+\vec B_{\vec v}\right)^{-1}}_{2}
    \le \frac{\norm{\vec B_{\vec v}}}{1-\norm{\vec B_{\vec v}}}
    \le \frac{\norm{\vec v}_{2}}{\tau - \norm{\vec v}_{2}},
  \end{equation*}
  which implies the first inequality of the theorem
  \begin{equation}
    \label{eq:8}
    \norm{\dx P_M(\vec m+\vec v)-\dx  P_M(\vec z)}_{2}
    \le \frac{1}{\tau} \norm{\vec m - \vec z}_{2}
    + \frac{1}{\tau- \norm{\vec v}_{2}} \norm{\vec v}_{2}.
  \end{equation}
  Since $\norm{\vec v}_{2}<\tau$ the point $\vec m + \vec v$ is within the tubular
  neighborhood of $\mathcal M$ and, hence
  $\norm{\vec v}_2 \leq \norm{\vec m+\vec v-\vec z}_2$. Together with the
  triangle inequality this gives us
  $\norm{\vec m -\vec z}_2\leq 2 \norm{\vec m +\vec v-\vec z}_2 $.  Including
  these two inequalities into \eqref{eq:8} we obtain the assertion.
\end{proof}

We observe that the constants in Theorem~\ref{thm:Psi_bound} become large if
either the reach of the manifold becomes small or the point $\vec m + \vec v$
is close to the boundary of the tubular neighborhood of $\M$.

\section{Manifold-valued Approximation}
\label{sec:manifold}

In this section we generalize arbitrary approximation operators for vector
valued functions to approximation operators for manifold-valued functions.
To this end we consider for an arbitrary domain $\Omega$
a generic approximation operator
$I_{\R^d} \colon C(\Omega,\R^d)\rightarrow
C(\Omega,\R^d)$. For an embedded manifold
$\M \subset \R^{d}$ with reach $\tau$ and projection operator
\begin{equation*}
  P_{\M} \colon U \to \M, \quad
  U = \{\vec y \in \R^{d} \mid \min_{\vec m \in \M} \norm{\vec y - \vec m} < \tau\},
\end{equation*}
we define the approximation operator
$I_{\M} \colon C(\Omega,\M) \to C(\Omega,\M) $ for manifold-valued functions as
\begin{equation*}
  I_\M f = P_\M\circ I_{\R^d}f.
\end{equation*}
It is important to note that $I_{\M}$ is not defined for all functions
$f \in C(\Omega,\M)$, but only for those for which $I_{\R^{d}} f(x)$
is within the reach of the manifold $\M$, i.e., $\norm{f(\vec x) - I_{\R^{d}}
  f(x)}_{2}\le \tau$ for all $\vec x \in \Omega$.

It is straight forward to see that operator $I_{\M}$ has the same order of
approximation as $I_{\R^{d}}$, c.f.~\cite{Ga18}.

\begin{theorem}\label{satz:wichtig}
  Let $f\in  C(\Omega,\M)$ be a continuous $\M$-valued function such that for
  all $\vec x \in \Omega$, $I_{\R^d}f(\vec x)$ is contained in the reach of
  $\M$. We then have for all $\vec x\in \Omega$
  \begin{equation}\label{eq:err_func}
    \norm{I_\M f(\vec x) - f(\vec x)}_2
    \leq  2\,\norm{I_{\R^{d}} f(\vec x) - f(\vec x)}_2.
  \end{equation}
\end{theorem}
\begin{proof}
  Since $f$ has function values on $\M$, it follows from the definition of
  $P_\M$ in equation~\eqref{eq:2} for all $\vec x\in \Omega$ that
  \begin{equation*}
    \left\|I_\M f(\vec x) - I_{\R^{d}} f(\vec x)\right\|_2
    \leq \left\|f(\vec x) - I_{\R^{d}} f(\vec x) \right\|_2.
  \end{equation*}
  Because of the triangle inequality and the definition of $I_{\M}$ we have
  \begin{equation*}
    \norm{I_\M f(\vec x) - f(\vec x)}_2
    \leq \norm{I_\M f(\vec x) - I_{\R^{d}} f(\vec x)}_2
    + \norm{I_{\R^{d}} f(\vec x) - f(\vec x)}_2,
    \quad \vec x \in \Omega.\qedhere
  \end{equation*}
\end{proof}

As we will see later, considering the error of the differential, things become more complicated.

\subsection{Approximation Order of the Differential}
\label{sec:app_diff}

In this section we are interested in the approximation error
$\norm{\d I_{\M} f - \d f}_2$ between the differential of the manifold-valued
approximation $\d I_{\M} f$ and the original differential $\d f$. To this end
we assume from now on that both, $f \colon \Omega \to \R^{d}$ and the vector-valued
approximation $\tilde f = I_{\R^{d}} f$, are differentiable.

While the error bound for $I_{\M} f$ is independent of the geometry of the
manifold $\M$, we will see that this is not true for the
differential $\d I_{\M} f$ of the manifold-valued approximation. Moreover, it
is not sufficient to ensure that $\tilde f$ is contained in the reach of
$\M$, but instead, it must be bounded away from the reach by some positive
constant.

\begin{theorem}\label{satz:wichtigableitung}
  Let $\tau$ be the reach of the manifold $\mathcal M$, $\varepsilon<\tau$ and
  $f\in C^1(\Omega, \M)$, such that $\tilde{f}(\vec x) = I_{\R^{d}} f(\vec x)$
  satisfies for all $\vec x\in \Omega$,
  \begin{equation*}
    \norm{f(\vec x)-\tilde{f}(\vec x)}_2\leq \varepsilon
  \end{equation*}
  and, consequently, is contained in the $\varepsilon$-tubular neighborhood of
  $\M$. Then we have for all $\vec x \in \Omega$ the following upper bound on the
  approximation error of the differential $\d I_\M f$,
	\begin{equation}
    \label{eq:error_diff}
    \begin{split}
      \norm{\d I_\M f(\vec x) - \d \,f(\vec x)}_{2}
      \leq \norm{\d \tilde{f}(\vec x) - \d f(\vec x)}_{2} + C \norm{f(\vec x)-\tilde{f}(\vec x)}_2
    \end{split}
  \end{equation}
	
  where $C$ is given by
  \begin{equation*}
    C = \left(\frac 2\tau+\frac{1}{\tau-\varepsilon}\right)
       \left(\norm{\d \tilde{f}(\vec x) - \d f(\vec x)}_{2}
        +\norm{\d f(\vec x)}_{2}\right).
					\end{equation*}
\end{theorem}
\begin{proof}
  By the chain rule we obtain for all $\vec x \in \Omega$,
  \begin{equation*}
    \d I_\M f(\vec x)=\dx P_\M(\tilde{f}(\vec x))\circ\d\tilde{f}(\vec x)
  \end{equation*}
  and from $P_\M f= f$,
  \begin{equation*}
    \d f(\vec x)
    =  \d (P_\M f)(\vec x)=\dx P_\M (f(\vec x))\circ \d f(\vec x).
  \end{equation*}
  Using the expansion
  \begin{align*}
    \d (I_\M f)(\vec x) - \d f(\vec x)
    =&\dx P_\M(\tilde{f}(\vec x))\circ \d \tilde{f}(\vec x)-
      \dx P_\M (f(\vec x))\circ \d f(\vec x)\\
    =& (\dx P_\M(\tilde{f}(\vec x))-\dx P_\M (f(\vec x)))
      \circ \d \tilde{f}(\vec x)\\
      &+ \dx P_\M f(\vec x) \circ
      ( \d\tilde{f}(\vec x)-\d f(\vec x))
  \end{align*}
  we conclude that
  \begin{align*}
    \norm{\d I_\M f(\vec x) - \d f(\vec x)}_2
    \le& \norm{\dx P_\M(\tilde{f}(\vec x))-\dx P_\M (f(\vec x))}_2
      \norm{\d \tilde{f}(\vec x)}_2\\
     &+ \norm{\dx P_\M (f(\vec x))}_2
      \norm{\d\tilde{f}(\vec x)-\d f(\vec x)}_2.
  \end{align*}
  Since $f(\vec x) \in \M$ and $\norm{f(\vec x) - \tilde f(\vec x)}_2 <
  \varepsilon$ we have by Theorem~\ref{thm:Psi_bound}
  \begin{equation*}
    \norm{\dx P_\M(\tilde{f}(\vec x))-\dx P_\M (f(\vec x))}
    \le \left(\frac 2\tau+\frac{1}{\tau -\varepsilon}\right)
    \norm{f(\vec x)-\tilde{f}(\vec x)}_2.
  \end{equation*}
  Together with the fact that $\norm{\dx P_\M f(\vec x)}_2=1$ we obtain
  \begin{align*}
    \norm{\d (I_\M f)(\vec x) - \d f(\vec x)}_2
    \le &\norm{f(\vec x)-\tilde{f}(\vec x)}_2
    \left(\frac 2\tau+\frac{1}{\tau -\varepsilon}\right)
    \norm{\d \tilde{f}(\vec x)}_2\\
    &+ \norm{\d\tilde{f}(\vec x)-\d f(\vec x)}_2.
  \end{align*}
  This implies the assertion by triangle inequality.
\end{proof}
Asymptotically, as $\norm{f(\vec x)-\tilde{f}(\vec x)}_2 \to 0$ we
  have $C \to \frac{3}{\tau} \norm{\d f(x)}_{2}$. Since for most approximation
  methods the decay of the differential
  $\norm{\d f(\vec x)-\d \tilde{f}(\vec x)}_2$ is one order slower than the
  decay of $\norm{f(\vec x)-\tilde{f}(\vec x)}_2$, we conclude that the right
  hand bound in \eqref{eq:error_diff} is dominated by the first summand and,
  hence, the approximation error of the differential $\d I_\M f$ of the
  manifold-valued approximant coincides asymptotically with the approximation
  error of the vector-valued approximant, as it was already reported in
  \cite{Ga18}.  However, the pre-asymptotic behavior depends strongly on the
  reach of the embedding of the manifold $\mathcal M$.

\subsection{Fourier Interpolation}
\label{sec:four-interp}

In this section we want to illustrate Theorem~\ref{satz:wichtigableitung}
using Fourier-Interpolation as the approximation operator $I_{\R^{d}}$.
More precisely, we define for a function $f \in C(\T, \R^{d})$ on the torus
$\T$ the Fourier partial sum
\begin{equation*}
  I_{\R^{d}} f(t) = S_n f(t) = \sum_{k=-n}^n c_k(f) e^{2\pi i k t}
\end{equation*}
with the vector-valued Fourier coefficients
\begin{equation*}
  c_{k}(f) = \int_{0}^{1} f(x) e^{-2\pi i k x} \d{x}.
\end{equation*}

The Fourier-Interpolation satisfies the following well known approximation
inequalities, cf.~\cite{PlPoStTa18}.

\begin{theorem}
  \label{thm:Sn}
  Let $r\in \N$ with $r\geq2$ and $f\in C^r(\T,\R^d)$. Then
  \begin{align*}
    \norm{f(x)-S_nf(x)}_{2}
    &\leq \frac{\sqrt{2d}}{(2\pi)^r}\,\frac{\sqrt n}{n^r}\,\norm{f^{(r)}}_{L^2(\T),2},\\
    \norm{\d f( x)-\d (S_nf)(x)}_2
    &\leq \frac{\sqrt{2d}}{(2\pi)^{r-1}}\,\frac{\sqrt n}{n^{r-1}}\,\norm{f^{(r)}}_{L^2(\T),2},
  \end{align*}
 with the norm $\norm{f}_{L^2(\T),2}^{2} = \int_{\T} \norm{f(x)}_2^{2} \d x.$

\end{theorem}

\begin{proof}
  The first bound can be found in \cite[Theorem 4.3]{Co16}. The second bound
  follows from $\d (S_nf)=S_n(\d f)$ and the fact that the regularity of
  $\d f$ is one less than the regularity of $f$. The factor $\sqrt{d}$ comes
  from the fact that the function $f$ maps in the d-dimensional space.
\end{proof}

In \cite[Theorem 1.39]{PlPoStTa18} a similar bound for the $L^\infty(\T)$-norm
can be found. Using the Fourier partial sum operator as the approximation
operator $I_{\R^{d}}$ Theorem~\ref{satz:wichtigableitung} becomes the
following.

\begin{theorem}
  \label{thm:FPS}
  Let $\mathcal M \subset \R^{d}$ be a submanifold with reach $\tau>0$ and
  $f\in C^r(\T,\M)$ an $r \geq 2$ times differentiable function with values in
  $\mathcal M$. Let, furthermore, $\varepsilon < \tau$ be an auxiliary
  constant and the bandwidth $n$ of the Fourier partial sum $S_{n} f$ at least
  such that $\norm{f( x)-S_{n} f(x)}_{2} \le \varepsilon$ for all $x\in \T$, i.e.,
  \begin{equation}
    \label{eq:error_in_reach}
    n^{r-\tfrac 12} \ge \frac{C_{1} \norm{f^{(r)}}_{L^2(\T),2}}{\varepsilon}
    \text{ with }
    C_{1} = \frac{ \sqrt{2d}}{(2\pi)^r}.
  \end{equation}
  Then the projection $P_{\mathcal M} \circ S_{n} f$ of the Fourier partial
  sum satisfies for all $x \in \T$,
  \begin{align}
    \label{eq:error_fourier}
    \left\|P_{\mathcal M} \circ S_{n} f(x) - f(x)\right\|_2
    &\leq  2 \,C_{1} \,n^{\tfrac12 -r}\,\norm{f^{(r)}}_{L^2(\T),2},
  \end{align}
  whereas for its differential $\d ( P_{\mathcal M}
  \circ S_{n} f)$ we obtain
  \begin{equation}
    \label{eq:error_diff_fourier}
    \begin{split}
      \left\|\d ( P_{\mathcal M} \circ S_{n} f)(x) - \d f(x)\right\|_2
      &\leq
      2\pi\, C_{1}\, \,n^{\tfrac 32-r}\, \norm{f^{(r)}}_{L^2(\T),2}\\
      &+C_{2}\, n^{\tfrac 12-r} \,\norm{f^{(r)}}_{L^2(\T),2}^{2},
    \end{split}
  \end{equation}
  with the constant
  \begin{align*}
    C_{2} = C_{1}\,(\tfrac{2}{\tau} + \tfrac{1}{\tau - \varepsilon})\,(1 + 2\pi\, C_{1}^{2}\, n^{\frac32-r}).
  \end{align*}
	
\end{theorem}

\begin{proof}
  Together with Theorem~\ref{thm:Sn} condition~\eqref{eq:error_in_reach}
  ensures that
  \begin{equation*}
    \norm{f(x) - S_{n} f(x)}_{2} \le \varepsilon < \tau, \quad x \in \T
  \end{equation*}
  and, hence, $S_nf(x)$ has distance less than the reach to $\M$ for all
  $x\in \T$.  This allows us to apply Theorem~\ref{satz:wichtig} in
  conjunction with Theorem~\ref{thm:Sn} to conclude \eqref{eq:error_fourier}.

  For the approximation error of the derivative we have by
  Theorem~\ref{satz:wichtigableitung}
  \begin{equation*}
    \norm{\d ( P_{\mathcal M} \circ S_{n} f)(x) - \d f(x)}_{2}
    \le \norm{\d (S_{n} f)(x) - \d f(x)}_{2} + C \norm{S_{n} f(x) - f(x)}_{2}
  \end{equation*}
  with
  \begin{equation*}
    C = \left(\tfrac{\tau}2 + \tfrac{1}{\tau-\varepsilon}\right) \, \left(\norm{\d (S_{n} f)(x) -
      \d f(x)}_{2} + \norm{\d f(x)}_{2}\right).
  \end{equation*}
  Together with Theorem~\ref{thm:Sn} this yields
  \begin{equation*}
    \norm{\d ( P_{\mathcal M} \circ S_{n} f)(x) - \d f(x)}_{2}
    \le 2\pi\, C_{1} \,n^{\frac{3}{2}-r} \,\norm{f^{(r)}}_{L^2(\T),2} + C \,C_{1}\, n^{\frac{1}{2}-r}\,\norm{f^{(r)}}_{L^2(\T),2}
  \end{equation*}
  and
  \begin{align*}
    C &\le (\tfrac{\tau}2 + \tfrac{1}{\tau-\varepsilon})\,
        (2\pi\,C_{1}\, n^{\frac{3}{2}-r} \norm{f^{(r)}}_{L^{2}(\T),2} + \norm{\d f(x)}_{2})\\
    & \le (\tfrac{\tau}2 + \tfrac{1}{\tau-\varepsilon})\,
        (2 \pi\, C_{1} \,n^{\frac{3}{2}-r}  +1 )\norm{f^{(r)}}_{L^{2}(\T),2},
  \end{align*}
  where we have used the fact that
  $\norm{\d f(x)}_2 \le \norm{f^{(r)}}_{L^{2}(\T),2}$ for periodic functions. Setting
  \begin{equation*}
   C_{2} = C_{1} (\tfrac{\tau}2 + \tfrac{1}{\tau-\varepsilon}) (1+2 \pi \,C_{1}\,
  n^{\frac{3}{2}-r}  )
  \end{equation*}
  yields the assertion.

\end{proof}

Since for $n \to \infty$ we have
$C_{2} \to C_{1}(\frac{2}{\tau}+\frac{1}{\tau-\varepsilon})$.
Theorem~\ref{thm:FPS} states that the approximation order of the differential
of the manifold-valued Fourier partial sum operator coincides with the
approximation order of the vector-valued operator. In the
preasymptotic setting, however, also the second summand with the faster rate
$n^{\frac{1}{2}-r}$ is relevant. The constant of this second summand becomes
large if the point-wise approximation error is close to the reach $\tau$ of
the embedding.

\section{Examples}\label{sec:examples}
In this section we apply our findings to two real world examples of
manifold-valued approximation. Both examples are related to the analysis of
crystalline materials. In the first example we consider functions that relate
propagation directions of waves to polarization directions and in the second
example we consider functions that relate points within crystalline specimen
to the local orientation of its crystal lattice. Both examples have
  been realized using Matlab Toolbox \texttt{MTEX 5.8},
  cf.~\cite{MaHiSc11}. The corresponding scripts and data files can be found
  at
  \url{https://github.com/mtex-toolbox/mtex-paper/tree/master/manifoldValuedApproximation}.

\subsection{Wave Velocities}
\label{sec:wave-velocities}

In crystalline materials the propagation velocity and polarization direction
of waves is often isotropic, i.e., it depends on the propagation direction
relative to the crystal lattice. This posses an important issue in seismology
where one analyzes the distribution of earthquake waves in order to get a
deeper understanding of the core of the earth,
cf. \cite{mainprice:hal-00408321}. Each earthquake wave decomposes into a
p-wave and two perpendicular shear-wave components. The polarization vectors
of p-wave components as well as of the two s-wave components depend on the
propagation direction of the wave relative to the crystal,
cf. \cite{nye85, MaHiSc11}. Mathematically, the directional dependency of the
polarization directions from the propagation direction is modeled as function
\begin{equation*}
 f \colon \mathbb S^{2} \to \mathbb RP^{2}
\end{equation*}
from the two-sphere $\mathbb S^{2}$ into the two--dimensional projective space
$\mathbb RP^{2}$. Our goal is to approximate this function from finite
measurements $\vec y_{n} = f(\vec x_{n}) \in \mathbb RP^{2}$, $n=1,\ldots,N$.

To this end, we identify the two dimensional projective space $\mathbb RP^{2}$
with the quotient $\mathbb S^{2}/{\sim}$ with respect to the equivalence
relation $\vec x \sim -\vec x$ and consider the embedding
$\mathcal E \colon \mathbb S^{2}/{\sim} \to \R^{3 \times 3}$,
$\mathcal E(\vec x) = \vec x \vec x^{\top}$.  The reach of this embedding is
$\tau = \tfrac{1}{\sqrt{2}}$ as we show in the following lemma.

\begin{lemma}
  The two dimensional projective space $\mathbb RP^{2}$ embedded into the
  space of symmetric $3 \times 3$ matrices
  $\mathcal E(\mathbb S^{2}/\!\! \sim) \subset \R^{3 \times 3}$ has the reach
  $\tau = \frac{1}{\sqrt{2}}$.
\end{lemma}
\begin{proof}
  Following~\cite[Thm. 2.2]{aamari19}, we can estimate the reach by the
  following infimum
  \begin{equation}\label{eq:tau}
    \tau
    = \inf_{\vec x\neq \vec y\in \mathbb S^{2}/{ \sim}}
    \frac{\norm{\E(\vec x)-\E(\vec y)}_2^2}
    {2\, d(\E(\vec x)-\E(\vec y),T_{ \vec y}\mathbb S^{2}/{\sim})}.
  \end{equation}
  Since in our setting in both spaces, $\mathbb S^{2}/{ \sim}$ and
  $\R^{3\times 3}$ the metric is invariant with respect to the action of
  $\SO$, it suffices to take the infimum for $\vec y=\vec e_1=(1, 0, 0)^\top$.
  We define the other canonical basis vectors in $\R^3$ as
  $\vec e_2=(0, 1, 0)^\top$ and $\vec e_3=(0, 0, 1)^\top$ The tangent vector
  space in $\vec e_1$ is then given by these tangent vectors:
  \begin{equation*}
    T_{(1, 0,0)^\top}\mathbb S^{2}/{ \sim}
    = \lin\left\{\frac{1}{\sqrt 2}(\vec e_2\vec e_1^\top
      + \vec e_1\vec e_2^\top),\frac{1}{\sqrt 2}(\vec e_3\vec e_1^\top
      + \vec e_1\vec e_3^\top)\right\}.
  \end{equation*}
  Hence, we write
  $\vec v = \E(\vec x)-\E(\vec e_1)=\vec x\vec x^\top-\vec e_1\vec e_1^\top$ and therefore we
  can calculate the reach by
  \begin{align*}
    \tau
    &=\inf_{\vec x\neq \vec e_1\in \mathbb S^{2}/{\sim}}
      \frac{\norm{\vec v}_2^2}{2\, \norm{\vec v-\langle \vec v,\frac{1}{\sqrt
      2}\vec e_2\vec e_2^\top\rangle
      -\langle \vec v,\frac{1}{\sqrt 2}\vec e_3\vec e_3^\top\rangle}_2}\\
    &=\inf_{\vec x\neq \vec e_1\in \mathbb S^{2}/{\sim}}\frac{2-2\,x_1^2}{2\,\sqrt{2-2\,x_1^2-2x_1^2x_2^2-2x_1^2x_3^2}}\\
    &=\inf_{\vec x\neq \vec e_1\in \mathbb
      S^{2}/{\sim}}\frac{1-x_1^2}{\sqrt{2}\,\sqrt{(1-x_1^2)^2}}=\frac{1}{\sqrt 2},
  \end{align*}
  which finishes the proof.
\end{proof}

The calculation of the reach gives us the constants in
Theorem~\ref{satz:wichtig} and \ref{satz:wichtigableitung} for this specific
manifold.

Let $\mathcal E \circ f \colon \mathbb S^{2} \to \R^{3 \times 3}$ be the
embedded function. A common method, cf.~\cite{Mic13}, of approximating the
spherical function $\mathcal E \circ f$ is by linear combinations of spherical
harmonics $Y_{\ell,k}$, $\ell=0,\ldots,L$, $k = -\ell,\ldots,\ell$ up to a
fixed bandwidth $L$,
\begin{equation*}
  S_{L} (\mathcal E \circ f)(\vec x)
  = \sum_{\ell=0}^{L} \sum_{k=-\ell}^{\ell} c_{\ell,k} Y_{\ell,k}(\vec x),
\end{equation*}
where the coefficients $c_{\ell,k} \in \R^{3 \times 3}$,
$k=-\ell,\ldots,\ell$, $\ell=0,\ldots,L$ are elements of the embedding space.
This expansion in spherical harmonics coincides with the linear
  Fourier operator in Section~\ref{sec:four-interp} defined for functions over
  the sphere.  In our little example we simply assume that the measurement
points $\vec x_{n}$ together with some weights $\omega_{n}$ form a spherical
quadrature rule up to degree $2L$ which allows us to determine the Fourier
coefficients $c_{\ell,k}$, by
\begin{equation*}
  c_{k,\ell} = \sum_{n=1}^{N} \omega_{n} \mathcal E(\vec y_{n}) \overline{Y_{\ell,k}(\vec x_{n})}.
\end{equation*}

Fig.~\ref{fig:wave_a} displays the theoretical polarization directions of an
Olivine crystal in dependency of the propagation direction. We observe the
points of singularity, marked by the black squares. In order to
approximate this non-smooth function we fixed the bandwidth $L=8$
and used 144 Chebyshev quadrature nodes
$\vec x_{1},\ldots,\vec x_{144} \in \mathbb S^{2}$ as sampling points,
cf.~\cite{Gr_pointsquad}. These quadrature nodes are approximately
  equispaced and are displayed as red lines in Fig.~\ref{fig:wave_a}.

The approximated function $P_{\R P^{2}} \circ S_{8}(\mathcal E \circ f)$ is
depicted in Fig.~\ref{fig:wave_b} and shows good approximation with the
original function away from the singularity points. This is supported
  by a plot of the point-wise error
  $\norm{ f(\vec x) - P_{\R P^2}\circ S_8 (\E \circ f)(\vec x)}_2$ in
  Fig.~\ref{fig:wave_c}. Note that we measure the error in the Euclidean norm
  of the $9$-dimensional embedding space $\R^9$, which is for
  $\vec x,\vec y\in \R P^2 $ equal to the Frobenius norm
  $\norm{\vec x\vec x^\top -\vec y\vec y^\top}_F$. Compared to this, Figure~\ref{fig:wave_d} 
	shows the error of the linear approximation $\norm{ f(\vec x) - S_8 (\E \circ f)(\vec x)}$, 
	which is half of the error bound from~\ref{satz:wichtig}. Additionally, we marked the areas where the 
	residual is bigger than the reach blue. In this regions our Theorems are not 
	applicable, since there the assumption $I_{\R^d}f$ is within the reach is not met.

\begin{figure}[t]
  \centering
  \begin{subfigure}[T]{0.4\textwidth}
    \includegraphics[width=0.8\textwidth]{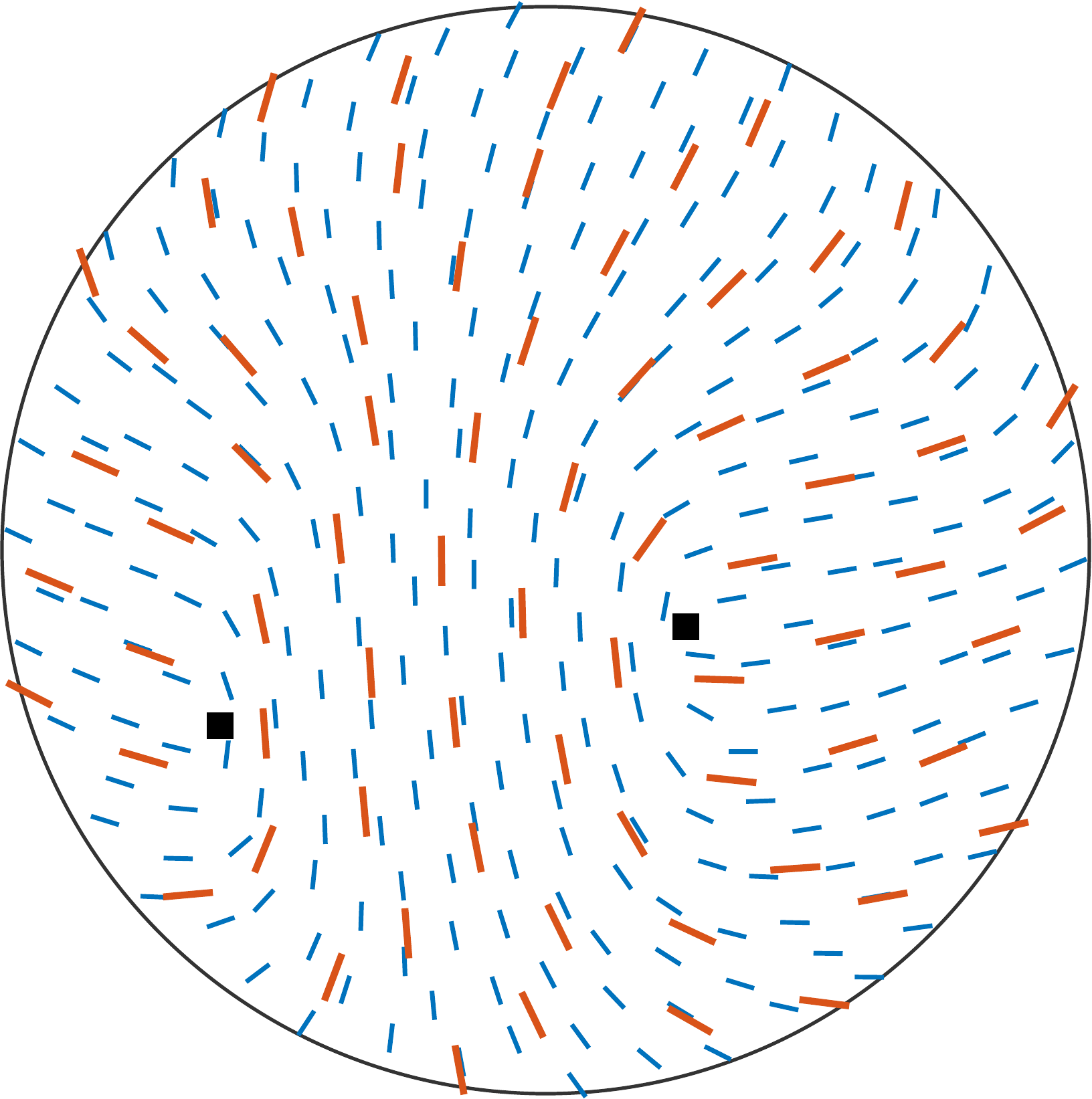}
		\centering
    \subcaption{$f \colon \sphere^{2} \to \R P^{2}$}
    \label{fig:wave_a}
  \end{subfigure}
  \begin{subfigure}[T]{0.4\textwidth}
    \includegraphics[width=0.8\textwidth]{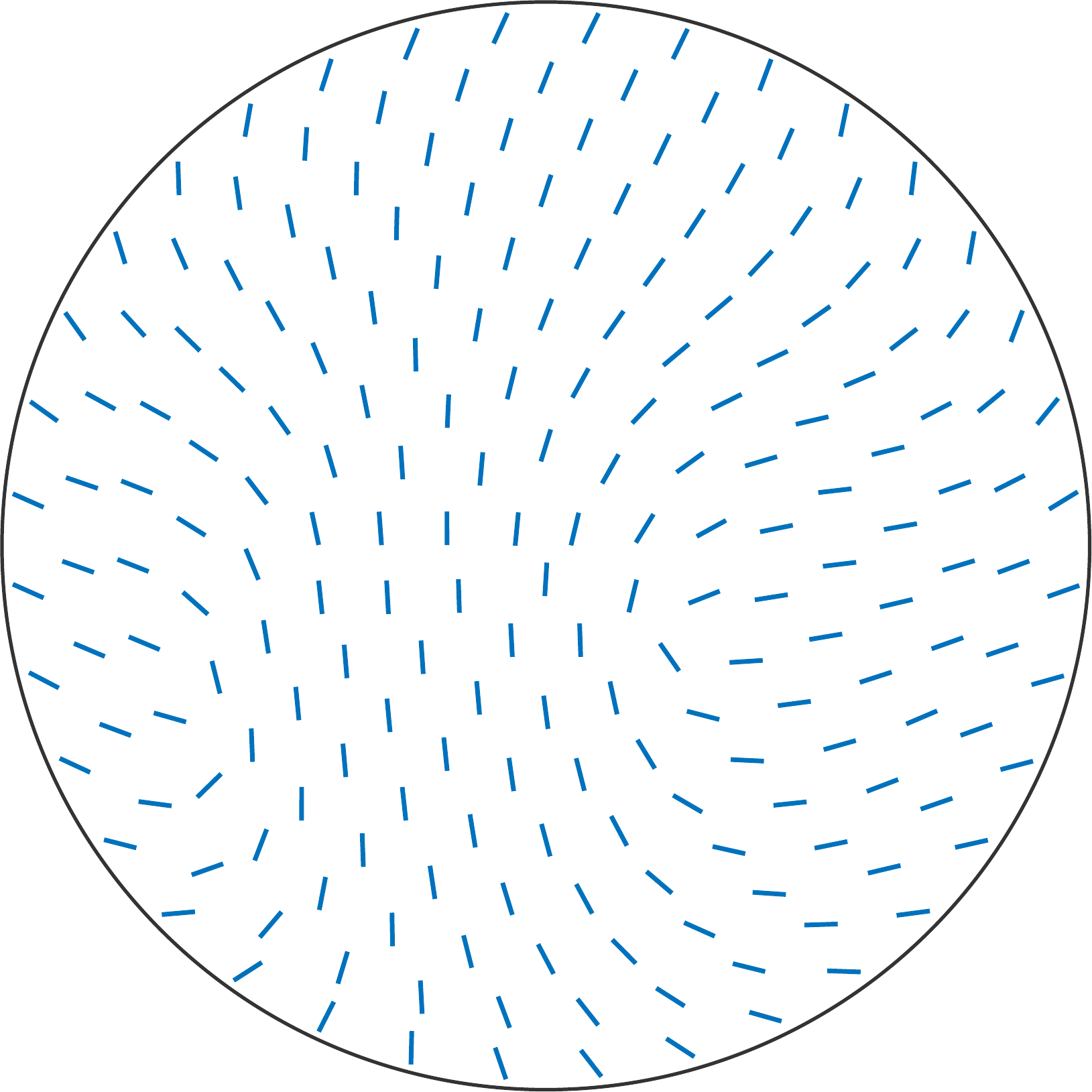}
		\centering
    \subcaption{$P_{\R P^2}\circ S_8 (\E \circ f)$}
    \label{fig:wave_b}
  \end{subfigure}
  \begin{subfigure}[T]{0.4\textwidth}
    \includegraphics[height=0.8\textwidth]{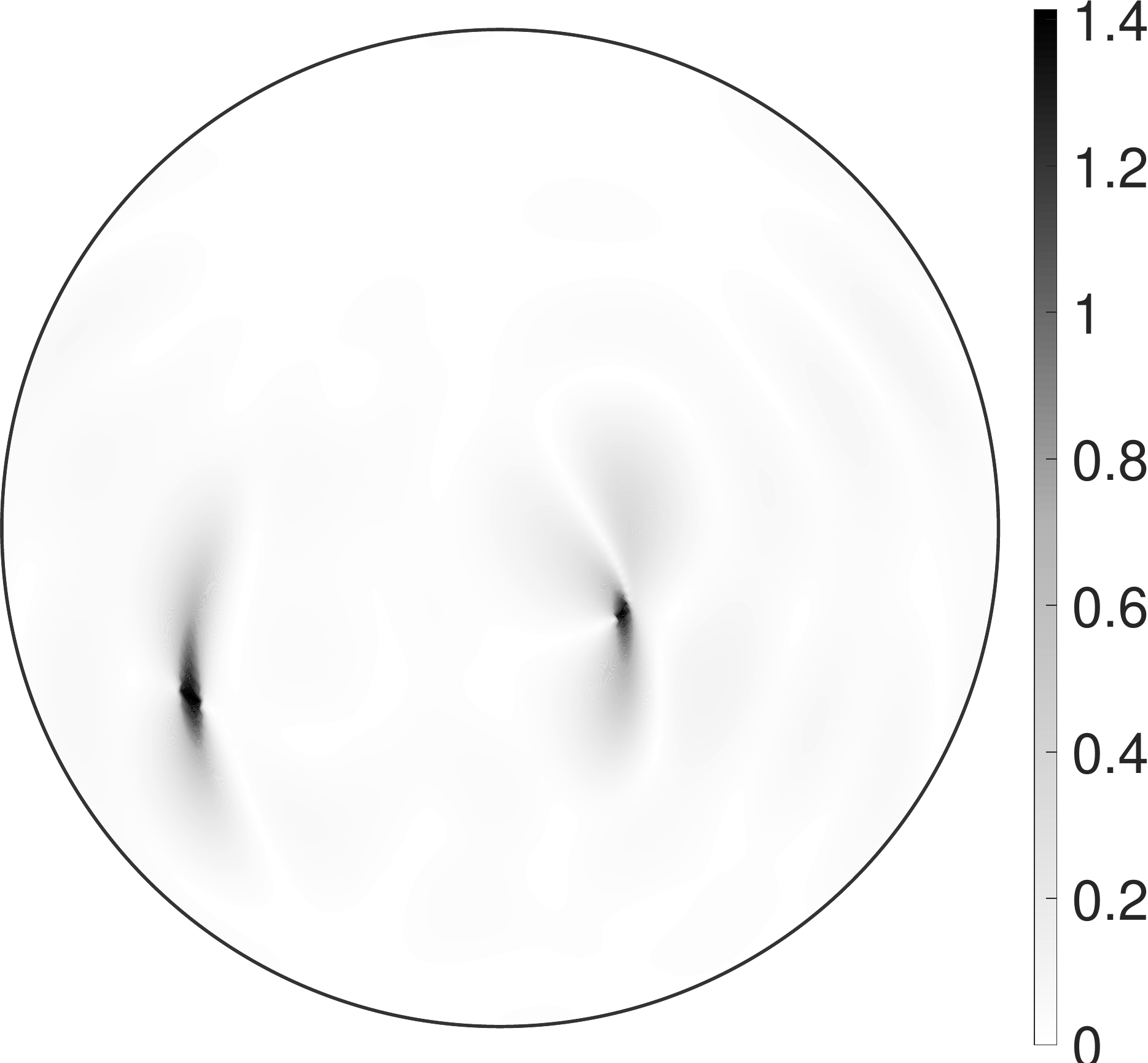}
		\centering
    \subcaption{$\norm{f-P_{\R P^2}\circ S_8 (\E \circ f)}_{2}$.}
    \label{fig:wave_c}
  \end{subfigure}
	  \begin{subfigure}[T]{0.4\textwidth}
    \includegraphics[height=0.8\textwidth]{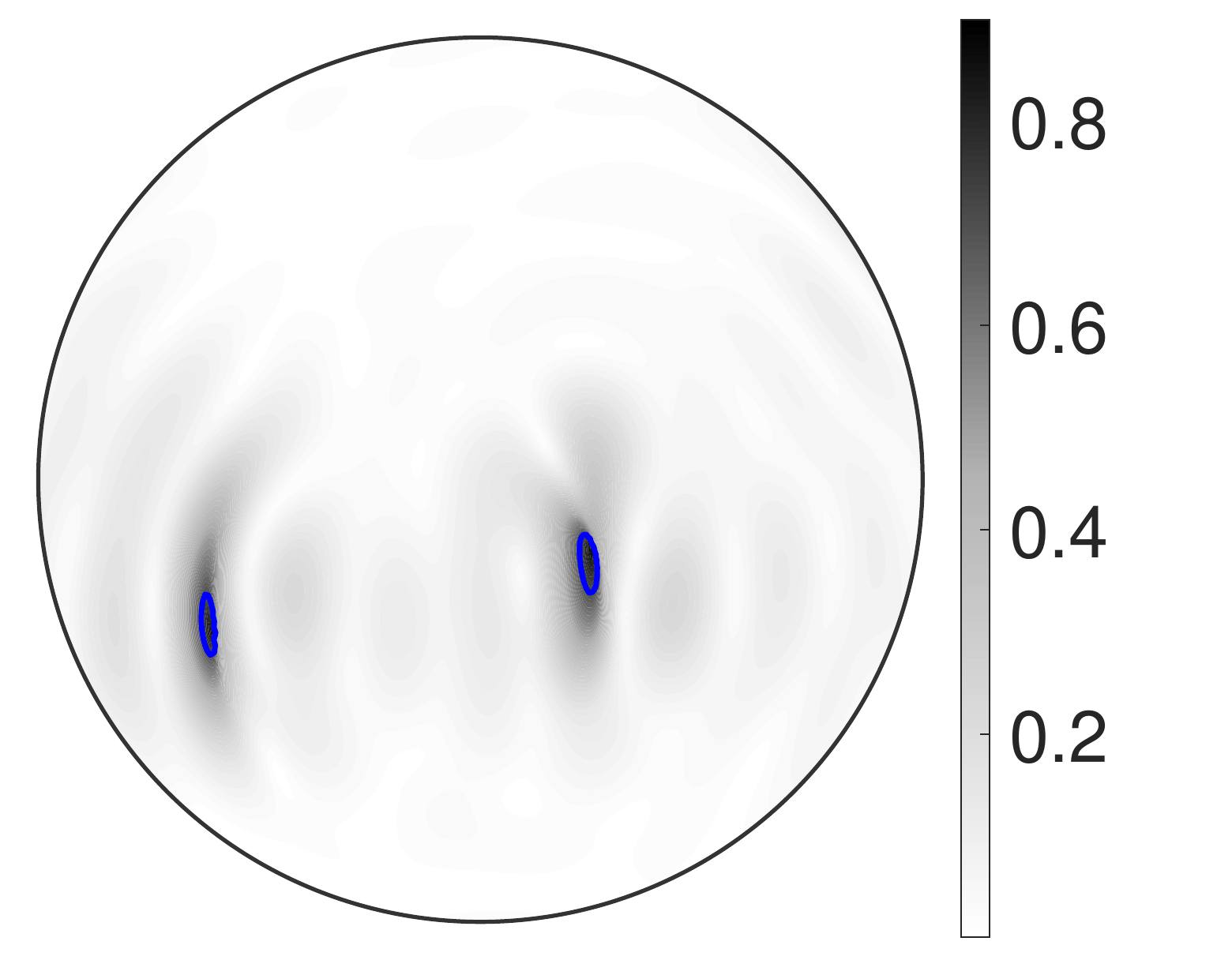}
		\centering
    \subcaption{$\norm{f- S_8 (\E \circ f)}_{2}$.}
    \label{fig:wave_d}
  \end{subfigure}
  \caption{Polarization directions of the fastest shear wave with
      respect to the propagation direction $\vec x \in \sphere^{2}$ plotted as
      vector fields on the upper hemisphere. The left upper image (a) displays the true
      polarization directions $f(\vec x) \in \R P^{2}$. The right upper image (b)
      is the harmonic approximation $P_{\R P^2}\circ S_8 (\E \circ f)(\vec x)$
      using the sampling points marked red in (a). The lower left image (c)
      displays the norm of the point-wise residual. The lower right image (d) 
			displays the error of the linear approximation, i.e. exactly half of the 
			upper bound from Theorem~\ref{satz:wichtig}. We marked red the regions 
			where this residual is bigger than the reach of the manifold.}
  \label{fig:wave}
\end{figure}

We determined the derivatives of $f$ numerically by choosing a basis
  ${\vec t_1,\vec t_2}$ in the tangent space $T_{\vec x}\S^2$ and
  approximating the columns $\vec v_i \in \R^{9}$, $i=1,2$ of
  $\d f(\vec x)\in \R^{9\times 2}$ by the difference quotients
  \begin{equation*}
    \vec v_i=\frac{1}{h} \left(f\bigl(\tfrac{\vec x+h\,\vec t_i}{\norm{\vec x+h\,\vec t_i}_2}\bigr)-f(\vec x)\right),
  \end{equation*}
  with $h=10^{-6}$. The norm of the derivative is depicted in
  Fig.~\ref{fig:dwave_a} and clearly shows the position of the singularities.
  In Fig.~\ref{fig:dwave_b} the error between $\mathrm d f(\vec x)$ and the
  differential of harmonic approximation
  $\d (P_{\R P^{2}} \circ S_{64})( \E \circ f)(\vec x)$ is plotted as a
  function of the propagation direction $\vec x\in \sphere^{2}$. Since the
  differential $\mathrm d f(\vec x) $ is a matrix in $\R^{9\times 2}$, we
  consider here the spectral norm of the error matrix. In order to illustrate
  our theoretical result of Theorem~\ref{satz:wichtigableitung} we plotted our
  theoretical upper bound on that approximation error of the derivative in
  Fig.~\ref{fig:dwave_c}. It should be noted that for the differential we
  needed to increased the polynomial degree to $L=64$ with $21000$ sample
  points in order to obtain a reasonable approximation at some distance to the
  singularities.

\begin{figure}[tb]

  \begin{subfigure}[T]{0.29\textwidth}
    \includegraphics[height=\textwidth]{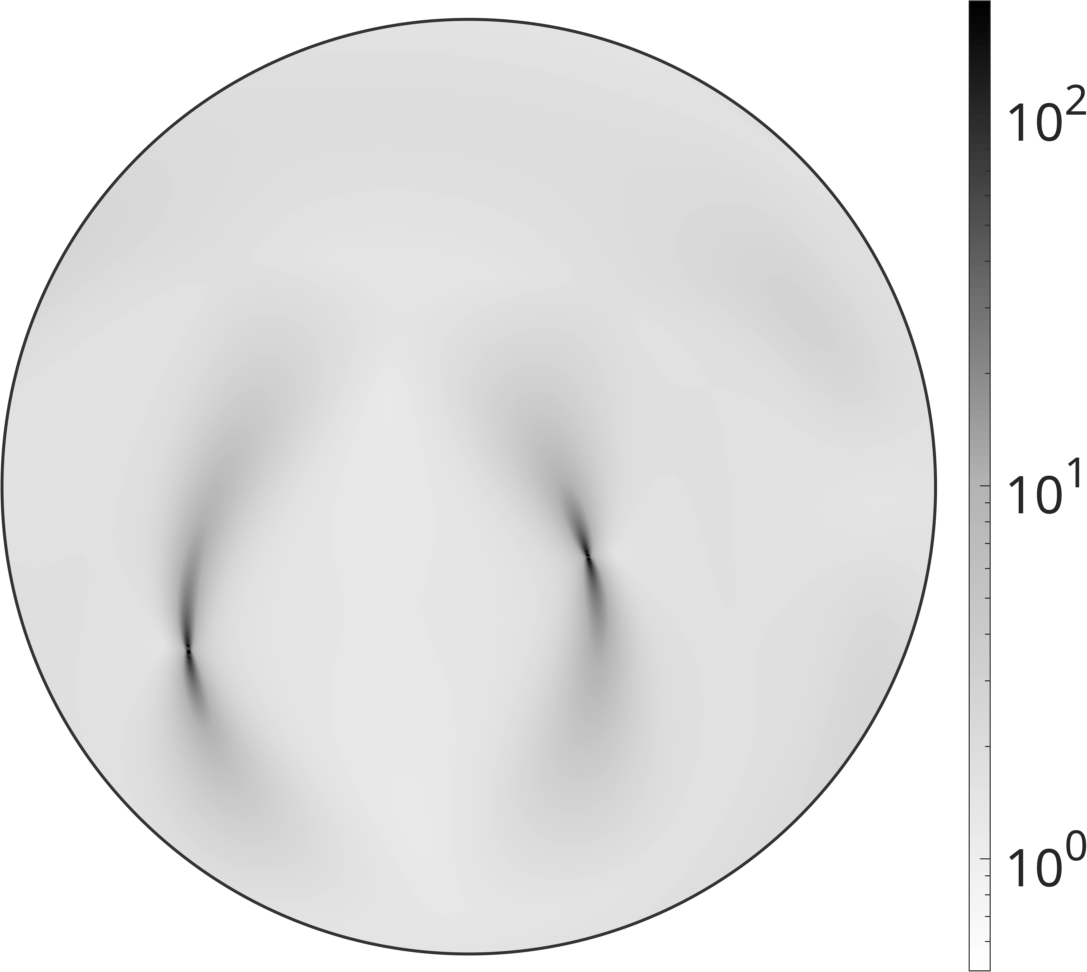}
    \subcaption{$\norm{d f(x)}_2$}
    \label{fig:dwave_a}
  \end{subfigure}
  \qquad
  \begin{subfigure}[T]{0.29\textwidth}
    \includegraphics[height=\textwidth]{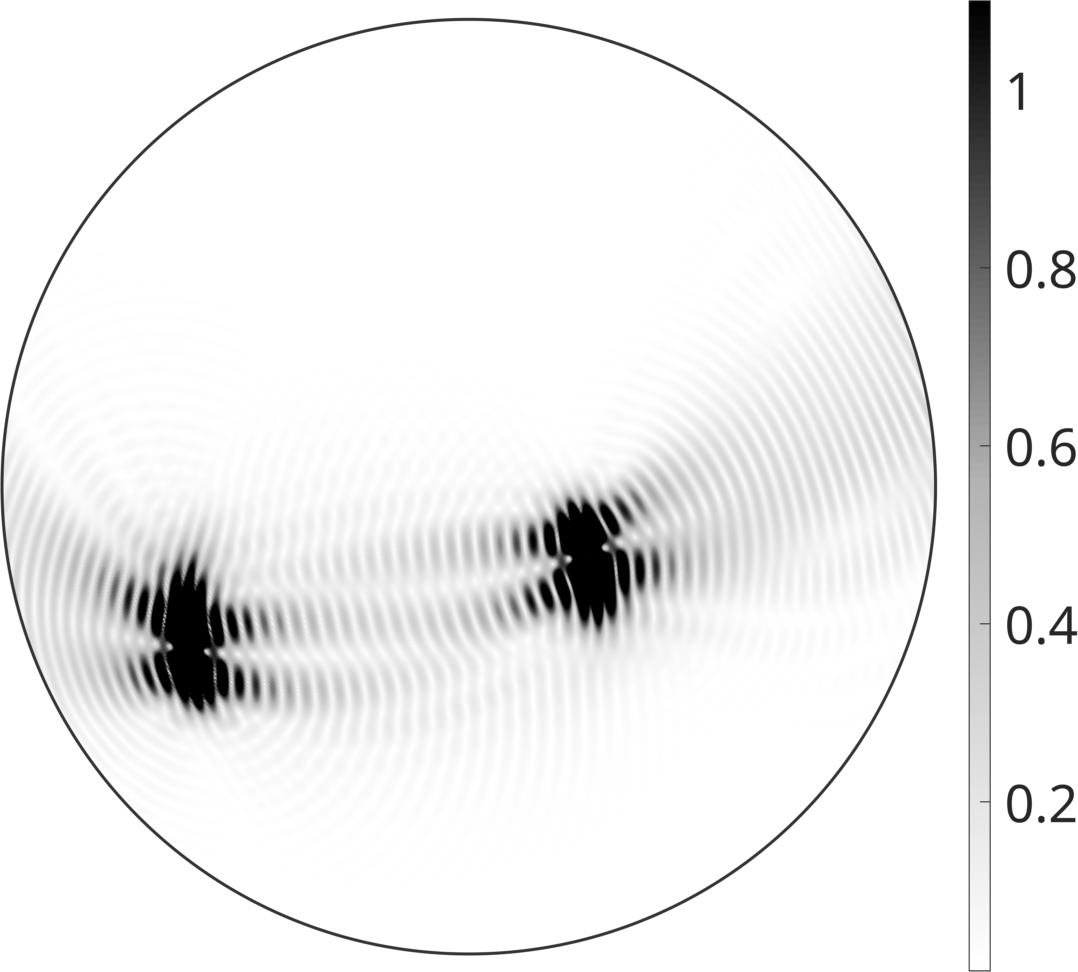}
    \subcaption{$\norm{d f(x) - d P_{\R P^2}\circ \tilde f}_2$}
    \label{fig:dwave_b}
  \end{subfigure}
  \qquad
  \begin{subfigure}[T]{0.29\textwidth}
    \includegraphics[height=\textwidth]{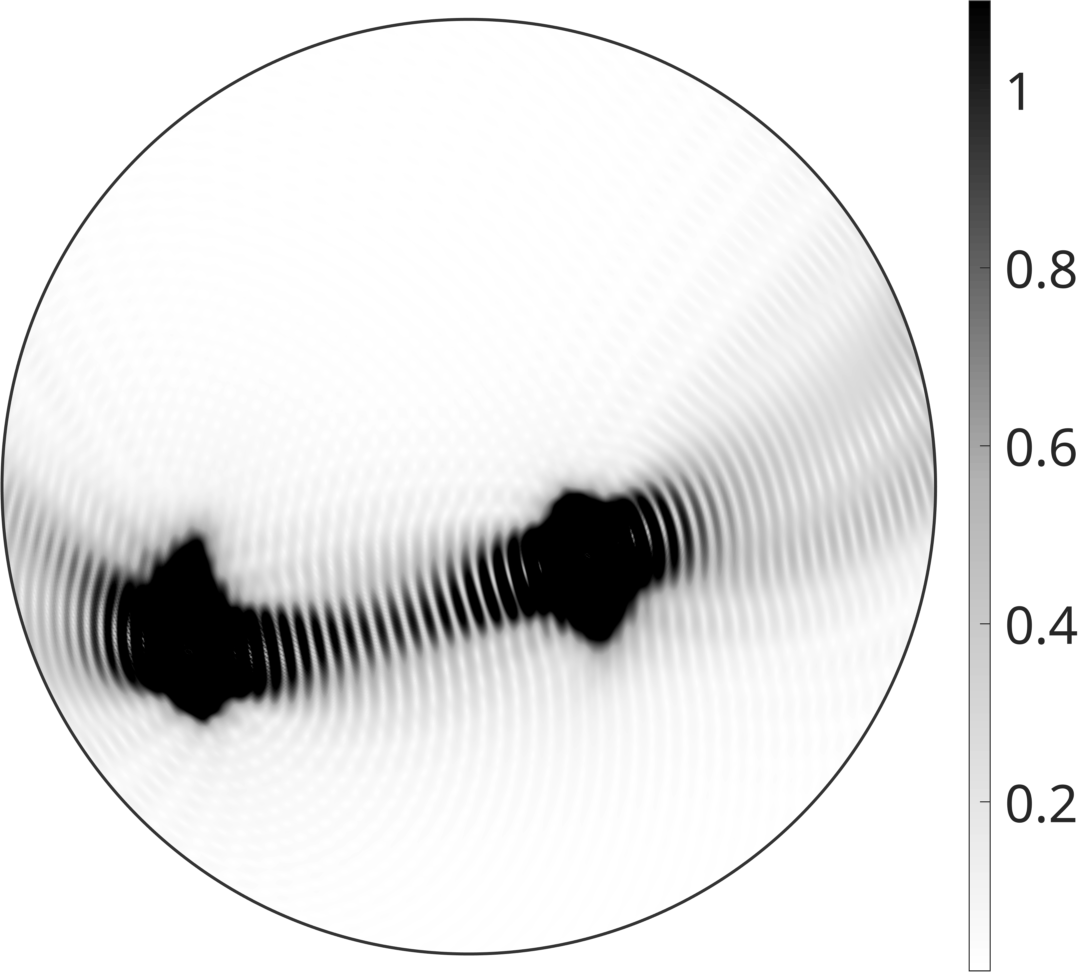}
    \subcaption{Bound from Thm.~\ref{satz:wichtigableitung}.}
    \label{fig:dwave_c}
  \end{subfigure}

  \caption{The left image (a) shows the point-wise norm of the
      differential $\d f$. The middle image (b) depicts the approximation
      error between the differential of the original function $f$ and the
      differential of its harmonic approximation
      $P_{\R P^{2}} \circ \tilde f = P_{\R P^{2}} \circ S_{64}(\mathcal E
      \circ f)$. The right images (c) gives the upper bound for (b) from
      Theorem~\ref{satz:wichtigableitung}.}
  \label{fig:dwave}
\end{figure}

\subsection{Electron Back Scatter Diffraction}
\label{sec:so3subset-r3times-3}

The subject of crystallographic texture analysis is the microstructure of
polycrystalline materials. Locally the microstructure is described by the
orientation of the atom lattice with respect to some specimen fixed reference
frame. More specifically, one describes the local orientation of the atom
lattice by a coset $[\vec R]_{\mathcal S} \in \SO / \mathcal S$ of the
rotation group $\SO$ modulo the finite subgroup $\mathcal S \subset \SO$,
called point group.  The point group of a crystal consists of all symmetries
of its atom lattice and is either one of the cyclic groups $C_{1}$, $C_{2}$,
$C_{3}$, $C_{4}$, $C_{6}$, the dihedral groups $D_{2}$, $D_{3}$, $D_{4}$,
$D_{6}$, the tetragonal group $T$ or the octahedral group $O$. Assuming a
monophase material, i.e., a material consisting only of a single type of
crystals, the variation of the local orientation of the atom lattice at the
surface $\Omega \subset \R^{2}$ of the specimen is modeled by the map
\begin{equation*}
  f \colon \Omega \to \SO/\mathcal S.
\end{equation*}

The gradient of the function $f$, also called lattice curvature tensor
$\vec \kappa(\vec x)$, is closely related to elastic and plastic deformations
the specimen has been exposed to. More specifically, it is related via the Nye
equation to the dislocation density tensor $\vec \alpha(\vec x)$,
\cite{Pa08,KoZaRa15} that describes how many lattice dislocations are
geometrically necessary in order to preserve the compatibility of the lattice
for a given deformation. Hence, estimating $f$ and its derivatives from
experimental data is a central problem in material science.

Electron back scatter diffraction (EBSD) is an experimental technique
\cite{ASWK93,KuWrAdDi93} for determining the local lattice orientations
$f(\vec x_{\ell}) \in \SO/\mathcal S$ at discrete sampling points
$\vec x_{i,j} \in \Omega$. An example of such EBSD data is the
$\SO/\mathcal S$ - valued image displayed in Fig.~\ref{fig:EBSDRAW}. It
describes the variation of lattice orientation at the surface of an Aluminum
alloy of size $200\mu m \times 150 \mu m$ at an resolution of $0.4\mu m$. The
symmetry group in this case is the octahedral group $O$.

\begin{figure}[t]
  \begin{subfigure}{0.42\linewidth}
    \includegraphics[height=3.9cm]{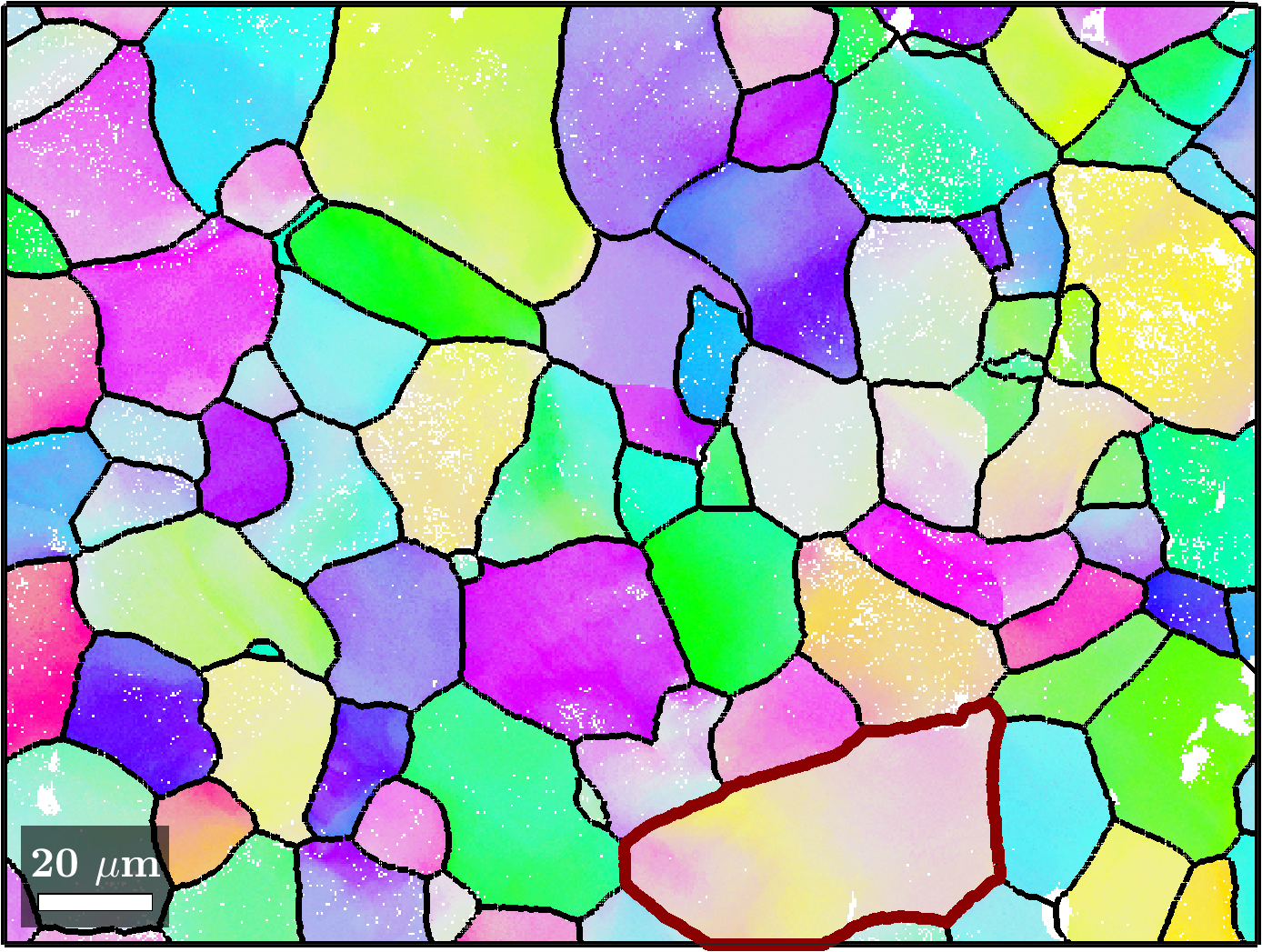}
    \subcaption{full map with a global color key}
    \label{fig:EBSDRAWa}
  \end{subfigure}
  \quad
  \begin{subfigure}{0.5\linewidth}
    \includegraphics[height=3.9cm]{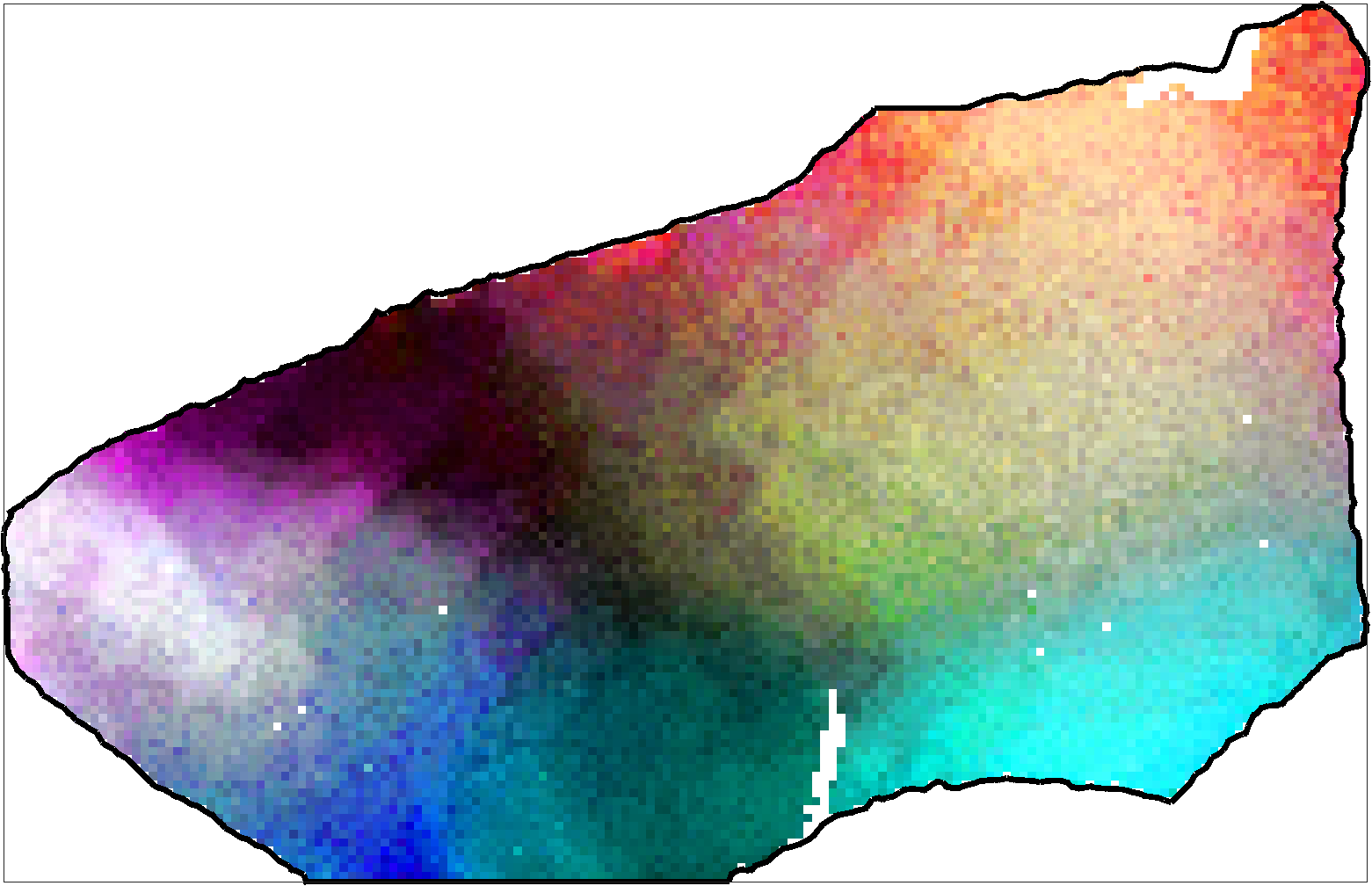}
    \subcaption{single grain with a local color key}
    \label{fig:EBSDRAWb}
  \end{subfigure}
  \centering
  \caption{The raw EBSD data. Each of the $410$ x $547$ pixels
    corresponds to a single orientation measurement at the surface of the
    specimen. The color is computed by the procedure described in
    \cite{NoHi17}. The $5\%$ white pixels in Fig. (a) correspond to
      corrupted measurements with no data. The data has been segmented into
      $92$ grains as outlined by the black boundaries.}
  \label{fig:EBSDRAW}
\end{figure}

The data is displayed with respect to two different color keys. In
Fig.~\ref{fig:EBSDRAWa} the colors are assigned globally to the cosets
$f(\vec x) \in \SO/O$ as described in \cite{NoHi17}. Regions of similar
lattice orientation form so-called grains as outlined by the black
boundaries. In Fig.~\ref{fig:EBSDRAWb} only the single grain outlined
  by the red boundary in Fig.~\ref{fig:EBSDRAWa} is displayed. For this grain
  we computed an average lattice orientation $[\vec M] \in \SO/O$ and selected
  for each coset $f(\vec x) \in \SO/O$ the rotation
  $\vec R(x) \in [\vec M^{-1} f(\vec x)]$ with the smallest rotational
  angle. Next we associated a color to $\vec R(x)$ according to a spherical
  color representation where the rotational angle of $\vec R(x)$ determines
  the saturation and the rotational axis hue and value. More details on this
  orientation coloring can be found in \cite{Thomsen17}.

Estimating the derivative from such a noisy map of lattice orientations is
usually not a good idea as we will see later.
Reducing the noise by means of local approximation methods has been discussed
in \cite{HiSiSc19,Seret19}. In order to demonstrate our embedding based
approximation approach we make use of the locally isometric embedding
$\mathcal E_{O} \colon \SO/O \to \R^{9}$ described in \cite{HiLi20} and
proceed as follows

\begin{enumerate}
\item Compute an $\R^{9}$-valued image $\vec u_{i,j} = \mathcal E_{O}(f(\vec
  x_{i,j}))$.
\item Approximate the $\R^{9}$-valued image using a cosine series
  $\tilde u \colon \Omega \to \R^{9}$ computed by robust, penalized least squares
  \cite{Gar10}.
\item Evaluate the function $\tilde{u}$ at the grid points $\vec x_{i,j}$ to obtain a
  noise reduced $\R^{9}$-valued image $\tilde{\vec u}_{i,j}$.
\item Compute the projection of $\vec{\tilde u}_{i,j}$ onto the
  embedding $\mathcal E_{O}(\SO/O)$ of the quotient and apply the inverse map
  $\mathcal E_{O}^{-1}$ to end up with a noise reduced $\SO/O$-valued image
  $\tilde f(\vec x_{i,j})$.
\end{enumerate}
The resulting image is depicted in Figure~\ref{fig:EBSDSmooth}. We
  observe that all no data pixels have been inpainted and that the magnified
  part~\ref{fig:EBSDSmoothb} is much less noisy in comparison to
  Fig.~\ref{fig:EBSDRAWb}.

\begin{figure}[t]
  \centering

 \begin{subfigure}{0.42\linewidth}
    \includegraphics[height=3.9cm]{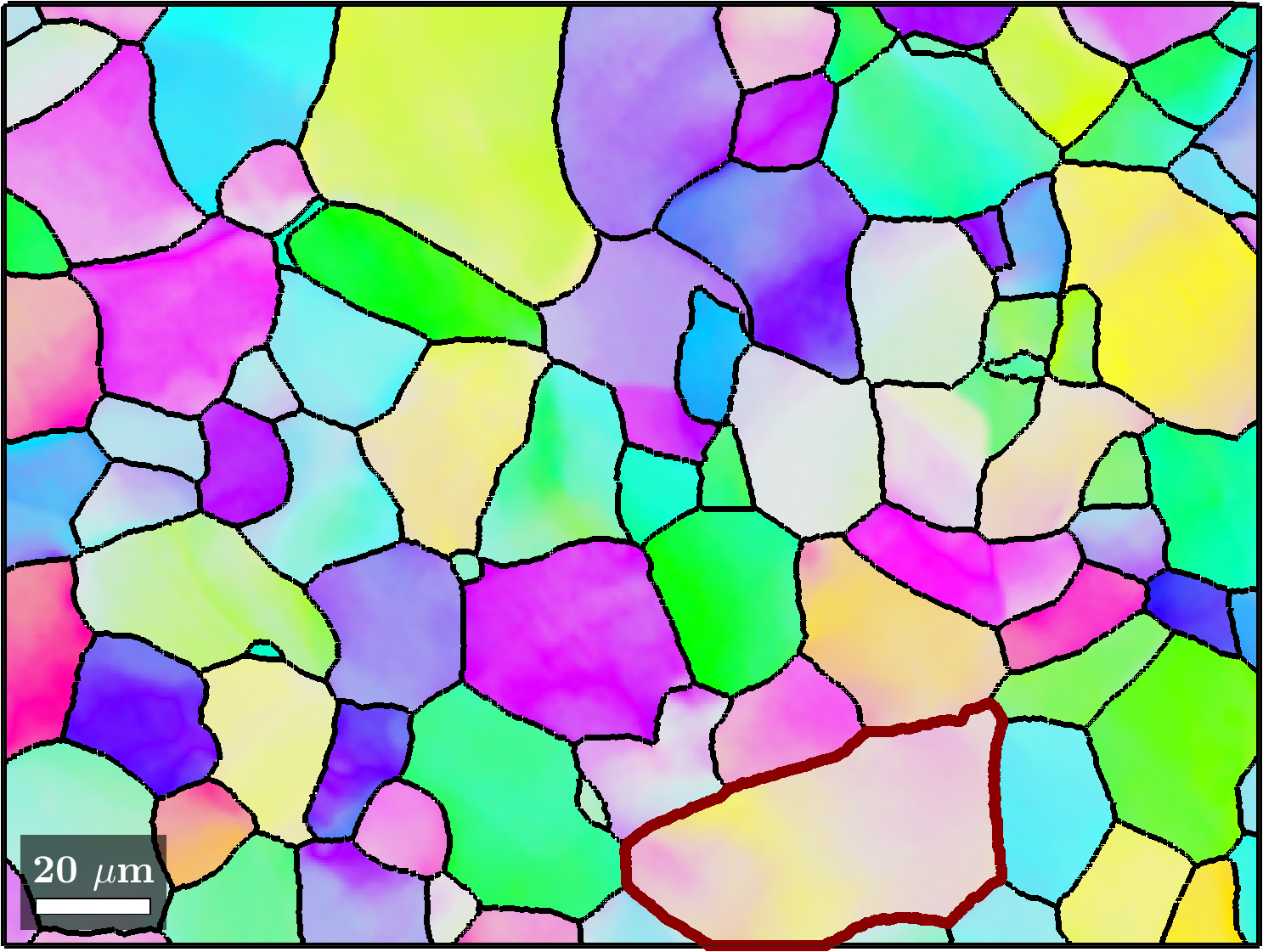}
    \subcaption{full map of the approximated data}
    \label{fig:EBSDSmootha}
  \end{subfigure}
  \quad
  \begin{subfigure}{0.5\linewidth}
    \includegraphics[height=3.9cm]{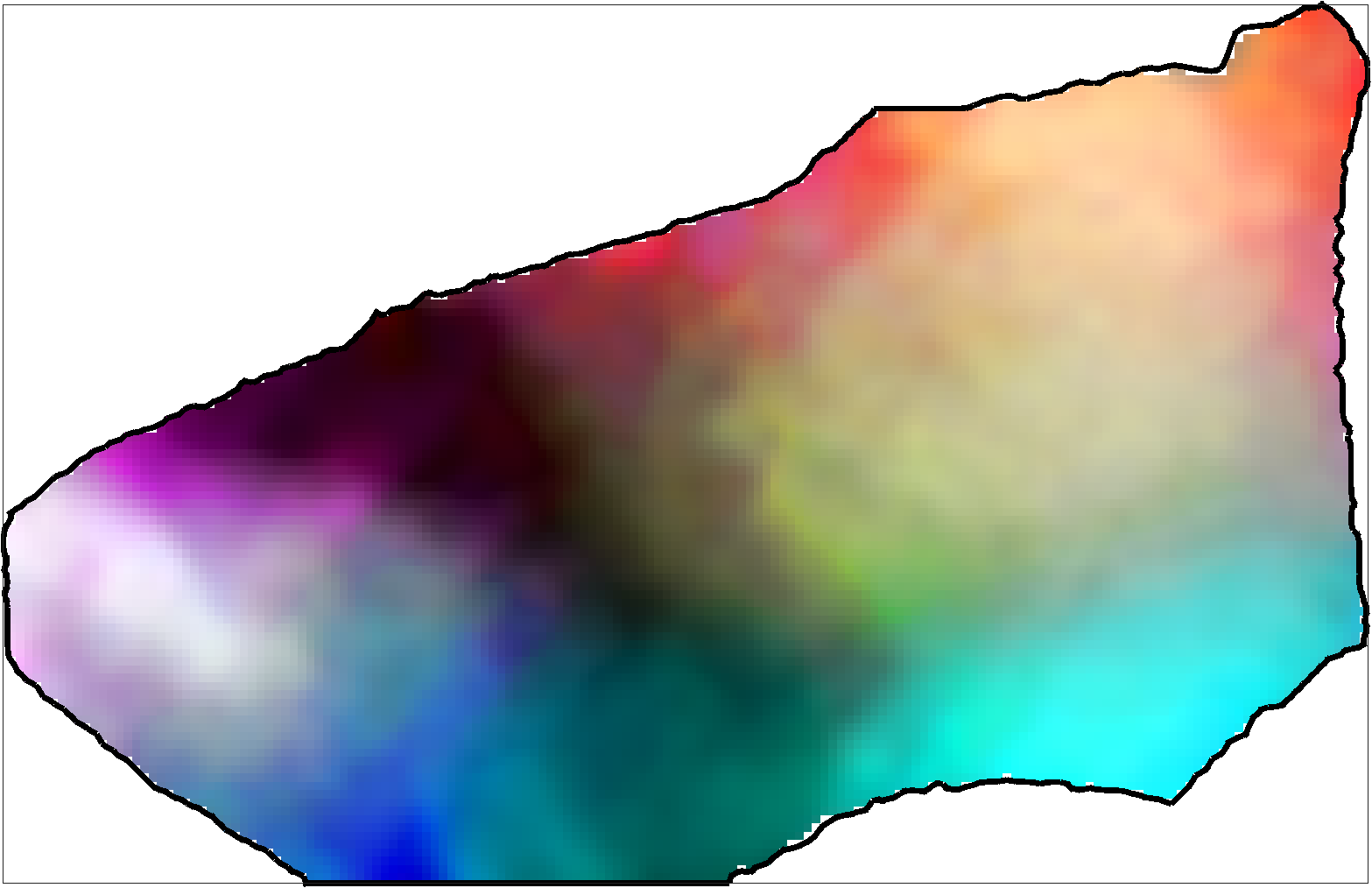}
    \subcaption{single grain with a local color key}
    \label{fig:EBSDSmoothb}
  \end{subfigure}
  \centering

  \caption{EBSD map from Fig.~\ref{fig:EBSDRAW} after the embedding based
    approximation approach.}
  \label{fig:EBSDSmooth}
\end{figure}

For the computation of the lattice curvature tensor $\vec \kappa$ we use
the skew symmetric matrices
\begin{equation*}
  \vec s^{(1)} =
  \begin{pmatrix}
    0&0&0\\
    0&0&-1\\
    0&1&0
  \end{pmatrix},\quad
  \vec s^{(2)} =
  \begin{pmatrix}
    0&0&-1\\0&0&0\\1&0&0
  \end{pmatrix},\quad
  \vec s^{(3)} =
  \begin{pmatrix}
    0&-1&0\\1&0&0\\0&0&0
  \end{pmatrix},
\end{equation*}
to fix the basis $\vec R \vec s^{(1)},\vec R \vec s^{(2)},\vec R \vec s^{(3)}$
in the tangential space $T_{\vec R} \SO/ O$ at some rotation
$\vec R \in \SO$.  With respect to this basis the differential
$D\mathcal E(\vec R) \colon T_{\vec R} \SO/O \to \R^{9}$ of the
embedding $\mathcal E \colon \SO/O \to \R^{9}$ can be represented as a full
rank $3 \times 9$ matrix. Furthermore, we obtain for the differential
$D \tilde u \colon \R^{2} \to \R^{9}$ of the embedded image
$\tilde u = \mathcal E \circ \tilde f \colon \Omega \to \R^{9}$ at some point
$\vec x \in \Omega$ the matrix product
$ D \tilde u(\vec x) = D \mathcal E(\tilde f(\vec x)) D f(\vec x)$. Hence, the
lattice curvature tensor $\tilde{\vec \kappa}$ of the noise reduced EBSD map
evaluates to
\begin{equation*}
  \tilde{\vec \kappa}(\vec x)
  = D \tilde f(\vec x)
  =  \left(D \mathcal E(\tilde f(\vec x)) D \mathcal E(\tilde f(\vec x))^{\top}\right)^{-1}
  D \mathcal E(\tilde uf(\vec x))^{\top} D \tilde u(\vec x).
\end{equation*}

The map of the first component $\tilde \kappa_{1,1}$ of the lattice curvature tensor
obtained from the approximating function $\tilde u$ is depicted in
Fig.~\ref{fig:EBSDKappaB}. For comparison we plotted in
Fig.~\ref{fig:EBSDKappaA} a finite difference approximation
\begin{equation*}
 \kappa(\vec x_{i,j})_{1,1} = \frac{\log_{f(\vec x_{i,j})}(f(\vec x_{i+1,j}) )}{[x_{i+1,j}-x_{i,j}]_{1}},
\end{equation*}
derived from the discrete data $f(\vec x_{i,j})$. Here, we denoted by
$\log_{\vec R} \colon \SO/O \to T_{\vec R} \SO/O$ the logarithmic
mapping with respect to the base point $\vec R \in \SO/O$. As expected, we observe that
the lattice curvature tensor $\tilde \kappa$ derived from the approximated map
$\tilde u$ is much less noisy.

\begin{figure}[htbp]
  \centering
  \begin{subfigure}{0.45\linewidth}
    \includegraphics[width=\textwidth]{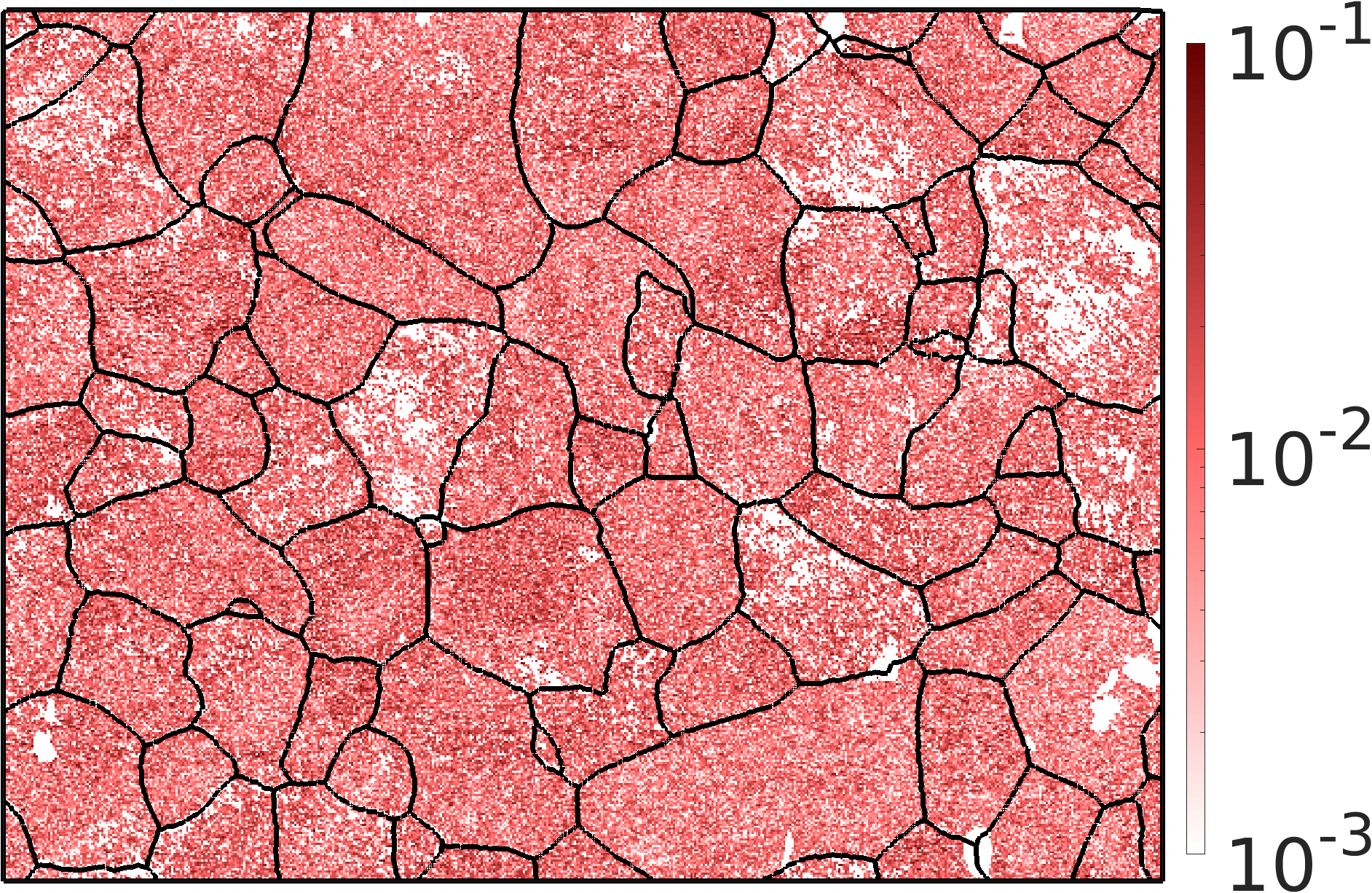}
    \caption{$\kappa_{1,1}$ of noisy map.}
    \label{fig:EBSDKappaA}
  \end{subfigure}
  \quad
  \begin{subfigure}{0.45\linewidth}
    \includegraphics[width=\textwidth]{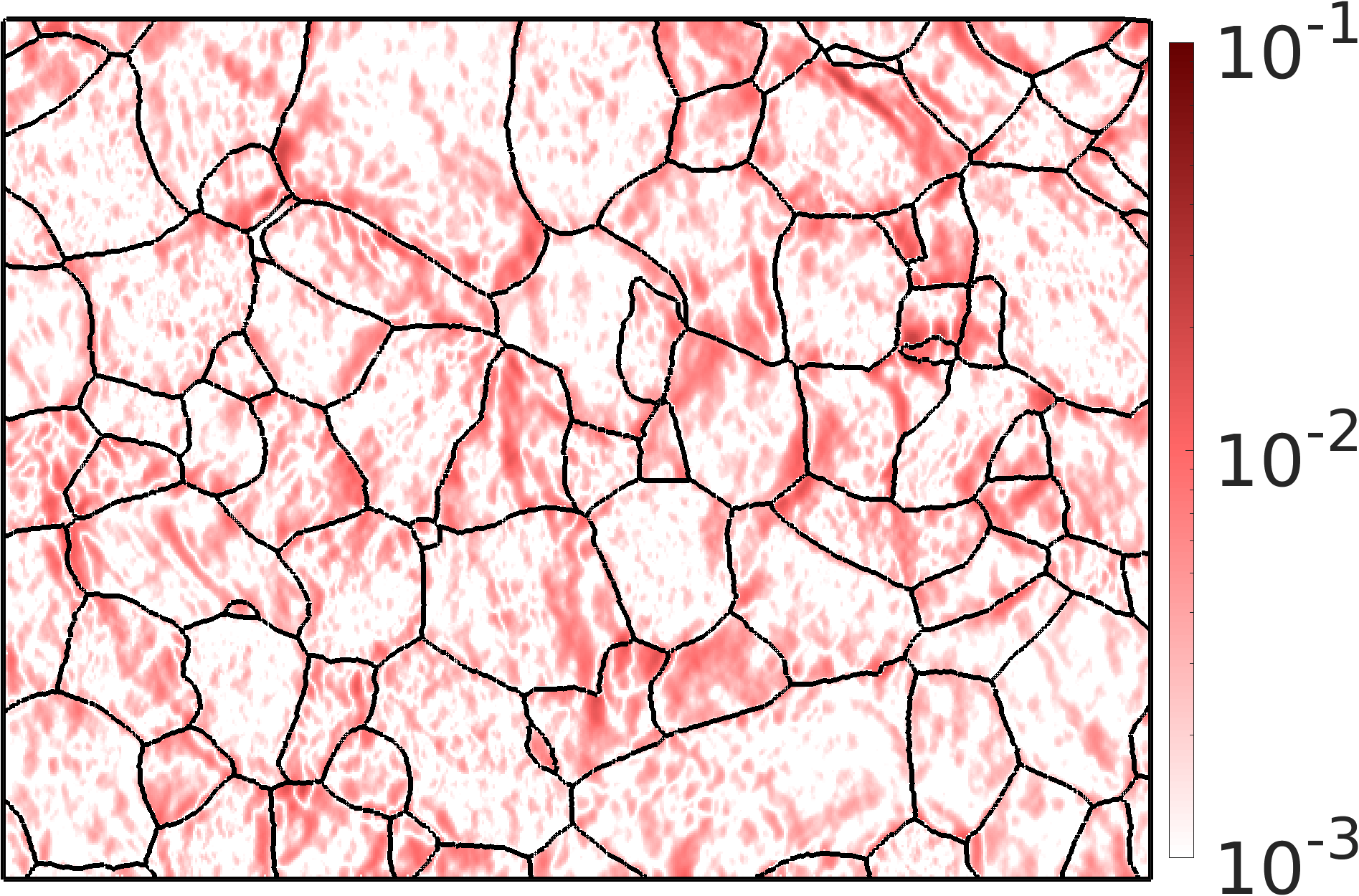}
    \caption{$\kappa_{1,1}$ of noise reduced map.}
    \label{fig:EBSDKappaB}
  \end{subfigure}

  \caption{First coefficient $\tilde\kappa_{1,1}(\vec x)$ of the lattice curvature
    tensor of the $SO(3)/O$-valued map depicted in Fig.~\ref{fig:EBSDRAW}
    (left) and Fig.~\ref{fig:EBSDSmooth} (right).}
  \label{fig:EBSDKappa}
\end{figure}
\newpage
\section*{Conclusion and further directions}
We proposed a method for approximating a manifold-valued function using an
embedding approach and a generic approximation operator $I_{\R^{d}}$ into the
Euclidean space. Our main result are the Theorems~\ref{satz:wichtig}
and~\ref{satz:wichtigableitung} which give upper bounds on the approximation
error for the function values as well as for the derivatives. The central
requirement of the Theorems is that generic approximation $I_{\R^d} f(\vec x)$
has distance less than $\tau$ to the manifold $\M$. For the approximation error of
the function values it is only important that $I_{\R^d} f(\vec x)$ is within
the reach of the manifold $\mathcal M$ and has no impact on the convergence
rate or the constants. However, for the approximation error of the derivatives
the constant of the upper bound becomes arbitrary large for $I_{\R^d} f(\vec
x)$ close to the reach. This stresses the importance of finding embeddings
with a large reach.

\medskip
The basis of our approximation approach is to find a suitable embedding of the
manifold $\M$ into $\R^d$.  For an arbitrary manifold this can be a difficult
challenge. However, for many important manifolds low dimensional embeddings
are well known. A further challenge is the numerical realization of the
projection operator $P_\M$ on the manifold needed for our embedding based
approximation method. This leads to a problem of manifold optimization.

\medskip

So far we did not deal with noisy data. The main challenge here is to guaranty
that $I_{\R^d} f(\vec x)$ is sufficiently close to the manifold even for noisy
data. Up to this point it is not clear how strongly noise that keeps the data on
the manifold can increase the distance of $I_{\R^d} f(\vec x)$ to
the manifold.

\bigskip

\section*{Acknowledgments}
The authors would like to thank Prof. Dr. Philipp Reiter for the nice hint for completing Theorem \ref{thm:C2Teil1}. 
Furthermore, we thank the anonymous reviewers for providing helpful comments and suggestions to improve this article.
The second author acknowledges funding by Deutsche Forschungsgemeinschaft (DFG, German Research Foundation) - Project-ID 416228727 - SFB 1410.

\appendix{}
\section{Bound for the commutator}
To bound the term $\norm{\d P_\M(\vec m)-\d P_\M(\vec z)}_2$ in section \ref{sec:change-proj-oper} we need a lemma, which is based on linear algebra.

\begin{lemma}
  \label{lemma:commutator}
  Let $\vec T$ be a projection matrix and $\vec R$ be a rotation matrix. Then
  there holds for the commutator
  \begin{equation*}
    \norm{\vec T\vec R-\vec R\vec T}_2\leq \norm{\vec I-\vec R}_2,
  \end{equation*}
  where again $\norm{\cdot}_2$ denotes the spectral norm.
\end{lemma}
\begin{proof}
  Since the spectral norm doesn't change under change of basis, we choose a
  matrix representation where the projection matrix has the form
  $$\vec T=\begin{pmatrix}\vec I&\vec 0\\ \vec 0&\vec 0\end{pmatrix},$$
  where $\vec I$ is the identity matrix of dimension $D$.  Then we also write the
  rotation matrix $\vec R$ in these blocks:
  \begin{equation*}
    \vec R=\begin{pmatrix}\vec R_{11}&\vec R_{12}\\ \vec R_{21}&\vec R_{22}\end{pmatrix}.
  \end{equation*}
  Simple matrix multiplication yields because of the orthogonality of $\vec R$
  \begin{align*}
    (\vec T\vec R-\vec R\vec T)^\top(\vec T\vec R-\vec R\vec T)
    &=\vec R^\top\vec T\vec R-\vec R^\top\vec T\vec R\vec T-\vec T\vec R^\top\vec T\vec R+\vec T\\
    &=\begin{pmatrix}\vec I-\vec R_{11}^\top\vec R_{11}&\vec 0\\ \vec 0&\vec
      R_{12}^\top\vec R_{12}
    \end{pmatrix}\\
    &=\begin{pmatrix}\vec R_{21}^\top\vec R_{21}&\vec 0\\ \vec 0&\vec
      R_{12}^\top\vec R_{12}\end{pmatrix}
  \end{align*}
  On the other hand there holds, again with help of the orthogonality of $\vec R$,
  \begin{align*}
    (\vec I-\vec R)^\top(\vec I-\vec R)
    &= 2\vec I-\vec R-\vec R^\top
      = \begin{pmatrix}2\vec I-\vec R_{11}-\vec R_{11}^\top&-\vec R_{12}-\vec R_{21}^\top\\ -\vec R_{21}-\vec R_{12}^\top&2\vec I-\vec R_{22}-\vec R_{22}^\top\end{pmatrix}\\
    &= \begin{pmatrix}(\vec I-\vec R_{11}^\top)(\vec I-\vec R_{11})+\vec I-\vec R_{11}^\top\vec R_{11}&-\vec R_{12}-\vec R_{21}^\top\\ -\vec R_{21}-\vec R_{12}^\top&(\vec I-\vec R_{22}^\top)(\vec I-\vec R_{22})+\vec I-\vec R_{22}^\top\vec R_{22}\end{pmatrix}\\
    &= \begin{pmatrix}(\vec I-\vec R_{11}^\top)(\vec I-\vec R_{11})+\vec R_{21}^\top\vec R_{21}&-\vec R_{12}-\vec R_{21}^\top\\ -\vec R_{21}-\vec R_{12}^\top&(\vec I-\vec R_{22}^\top)(\vec I-\vec R_{22})+\vec R_{12}^\top\vec R_{12}\end{pmatrix}.
  \end{align*}
  The spectral norm of a matrix $\vec A$, i.e., the largest absolute
  value of the eigenvalues can be written as
  \begin{equation*}
    \norm{\vec A}_2^2=\max_{\norm{\vec x}_2=1}\norm{\vec x^\top \vec A^\top \vec A\vec x}_2.
  \end{equation*}
  For that reason we choose the vector $\vec x$ as the eigenvector of the
  matrix $(\vec T\vec R-\vec R\vec T)^\top(\vec T\vec R-\vec R\vec T)$.  Since
  the eigenvalues and eigenvectors of a block-diagonal matrix are the union of
  the eigenvalues and eigenvectors, i.e., there holds
  $\vec x = \begin{pmatrix}\vec x_1 &\vec 0\end{pmatrix}^\top$ or
  $\vec x = \begin{pmatrix}\vec 0 &\vec x_2\end{pmatrix}^\top$.  We assume the
  first case, the other one is analog.  Hence, there holds
  \begin{equation*}
    \norm{\vec T\vec R-\vec R\vec T}_2^2
    =\vec x^\top\begin{pmatrix}\vec R_{21}^\top\vec R_{21}&\vec 0\\ \vec
      0&\vec R_{12}^\top\vec R_{12}\end{pmatrix} \vec x
    = \vec x_1^\top\vec R_{21}^\top\vec R_{21}\vec x_1.
  \end{equation*}
  If we look at the norm of the matrix $\vec I-\vec R$, we
  get
  \begin{align*}
    \norm{\vec I-\vec R}_2^2&\geq \vec x^\top \,(\vec I-\vec R)^\top(\vec I-\vec R)\,\vec x=\vec x_1^\top\,\left((\vec I-\vec R_{11}^\top)(\vec I-\vec R_{11})+\vec R_{21}^\top\vec R_{21}\right)\,\vec x_1\\
&\geq \vec x_1^\top\vec R_{21}^\top\vec R_{21}\vec x_1,
  \end{align*}
  since the eigenvalues of $(\vec I-\vec R_{11}^\top)(\vec I-\vec R_{11})$ are
  positive. Putting this together and taking the square root, yields the
  assertion.
\end{proof}

\bibliographystyle{abbrv}
\bibliography{../../references/references}

\end{document}